\documentclass[12pt]{amsart}
\usepackage{a4wide,enumerate,color}
\usepackage{scrextend} 
\allowdisplaybreaks

\let\pa\partial  
\let\na\nabla  
\let\eps\varepsilon   
  
\newcommand{\R}{{\mathbb R}} 
\newcommand{\diver}{\operatorname{div}}

\newcommand{\dom}{{\mathcal G}}
\newcommand{\T}{{\mathcal T}}

\newtheorem{theorem}{Theorem}   
\newtheorem{lemma}[theorem]{Lemma}   
\newtheorem{proposition}[theorem]{Proposition}   
\newtheorem{remark}[theorem]{Remark}   
  
\newtheorem{definition}{Definition}

\def\Xint#1{\mathchoice
{\XXint\displaystyle\textstyle{#1}}%
{\XXint\textstyle\scriptstyle{#1}}%
{\XXint\scriptstyle\scriptscriptstyle{#1}}%
{\XXint\scriptscriptstyle\scriptscriptstyle{#1}}%
\!\int}
\def\XXint#1#2#3{{\setbox0=\hbox{$#1{#2#3}{\int}$ }
\vcenter{\hbox{$#2#3$ }}\kern-.58\wd0}}

\def\dashint{\Xint-}

\begin{document}  

\title[Homogenization of cross-diffusion systems]{Homogenization of degenerate
cross-diffusion systems}

\author{Ansgar J\"ungel}
\address{Institute for Analysis and Scientific Computing, Vienna University of  
	Technology, Wiedner Hauptstra\ss e 8--10, 1040 Wien, Austria}
\email{juengel@tuwien.ac.at} 

\author{Mariya Ptashnyk}
\address{Department of Mathematics, School of Mathematical and Computer Sciences, Heriot-Watt University, EH14 4AS Edinburgh, 
Scotland, United Kingdom.}
\email{m.ptashnyk@hw.ac.uk} 

\date{\today}

\thanks{AJ acknowledges partial support from   
the Austrian Science Fund (FWF), grants P27352, P30000, F65, and W1245, and    
the Austria-Croatia Program of the Austrian Exchange Service (\"OAD)} 

\begin{abstract}
Two-scale homogenization limits of parabolic cross-diffusion systems in a
heterogeneous medium with no-flux boundary conditions are proved. The heterogeneity
of the medium is reflected in the diffusion coefficients or by the perforated
domain. The diffusion matrix is of degenerate type and 
may be neither symmetric nor positive semi-definite, 
but the diffusion system is assumed to satisfy an entropy structure. Uniform estimates
are derived from the entropy production inequality. New estimates on the
equicontinuity with respect to the time variable 
ensure the strong convergence of a sequence of solutions 
to the microscopic problems defined in perforated domains. 
\end{abstract}

\keywords{Periodic homogenization, strongly coupled parabolic systems,
two-scale convergence, perforated domain, entropy method}
 
\subjclass[2000]{35B27, 35K51, 35K59, 35K65}  

\maketitle


\section{Introduction}

Multicomponent systems are ubiquitous in nature; examples are as various as 
gas mixtures, bacterial colonies, lithium-ion battery cells, and animal crowds.
On the diffusive level, these systems can be described by cross-diffusion 
equations taking into account multicomponent diffusion and reaction \cite{Jue16}. 
When the mass transport occurs in a domain with periodic microstructure or
in a porous medium, macroscopic models can be derived from the microscopic  
description of the processes by homogenization techniques.
In this paper, we consider cross-diffusion systems
defined in a heterogeneous medium, where the heterogeneity is reflected in 
spatially periodic diffusion coefficients or by the perforated domain. 
The corresponding macroscopic equations
are derived by combining, for the first time, two-scale convergence techniques 
and entropy methods.

The problem of reducing a heterogenous material to a homogenous one has been
investigated in the literature since many decades. The research started in the
19th century by Maxwell and Rayleigh and was developed later by engineers leading
to asymptotic expansion techniques. Homogenization became a topic in mathematics
in the 1960s and 1970s. For instance, the $\Gamma$-convergence was introduced 
by De Giorgi \cite{DeG75} with the aim to describe the
asymptotic behavior of functionals and their minimizers.
The $G$-convergence of Spagnola \cite{Spa68} and its generalization to nonsymmetric 
problems, the $H$-convergence of Tartar and Murat \cite{MuTa77}, are related to the 
convergence of the Green kernel of the corresponding elliptic operator.
The two-scale convergence \cite{All92, Ngu89} combines formal asymptotic expansion 
and test function methods. Nguetseng introduced an extension of two-scale 
convergence to almost periodic homogenization, called $\Sigma$-convergence 
\cite{Ngu03,Ngu04}. Another extension concerns the two-scale convergence
in spaces of differentiable functions \cite{Vis04}, which is important
in nonlinear problems \cite{MiTi07}.
A classical reference for the homogenization theory of 
periodic structures is \cite{BLP78}. 

In spite of the huge amount of literature on homogenization problems, there
are not many studies on the homogenization of nonlinear parabolic systems.
Most of the results concern weakly coupled equations like periodic homogenization
of reaction-diffusion systems or of thermal-diffusion equations
in periodically perforated domains \cite{AlHu15,CHNN15,PtSe16}.
Particular cross-diffusion systems -- of triangular type -- were
investigated in \cite{KAM14}. However, up to our knowledge, there are no results
on more general cross-diffusion systems. 

In this paper, we investigate strongly coupled parabolic cross-diffusion systems
with a formal gradient-flow or entropy structure by combining two-scale convergence 
and the bounded\-ness-by-entropy method \cite{Jue15}. 
The difficulty is the handling of the degenerate structure of the equations.
We investigate two classes of degeneracies: a local one of porous-medium type
and a nonlocal one; see Section \ref{sec.main} for details.

The paper is organized as follows. In Section \ref{sec.main}, the microscopic
models are formulated and the main results are stated. 
The main theorems are proved in Sections \ref{sec.thm1} and \ref{sec_proof2}.
For the convenience of the reader, the definition and some properties
of two-scale convergence are recalled in Appendix \ref{app}. 
 The technical Lemma \ref{Lemma_entropy} is proved in Appendix \ref{sec.lemma}.
Finally, two cross-diffusion systems from applications which satisfy our assumptions
are presented in Appendix \ref{app.ex}.


\section{Formulation of the microscopic models and main results}\label{sec.main}

We investigate two types of homogenization problems. The first homogenization
limit is performed in cross-diffusion systems with spatially periodic coefficients, 
\begin{equation}\label{1.eq}
  \pa_t u_i^\eps - \diver\bigg(\sum_{j=1}^n
	P\bigg(\frac{x}{\eps}\bigg)a_{ij}(u^\eps)\na u_j^\eps\bigg)
	= f_i(u^\eps) \quad  \text{in } \Omega, \ t>0, \ i=1,\ldots,n,
\end{equation}
in a bounded domain $x\in\Omega\subset\R^d$ ($d\ge 1$), together with
no-flux boundary and initial conditions
\begin{equation}\label{1.bic}
  \sum_{j=1}^n P\bigg(\frac{x}{\eps}\bigg)a_{ij}(u^\eps)
	\na u_j^\eps\cdot\nu = 0\mbox{ on }\pa\Omega,\;  t>0, \quad
	u_i^\eps(0)=u_i^0  \; \mbox{ in }\Omega.
\end{equation}
Here, $u^\eps=(u_1^\eps,\ldots,u_n^\eps)$ is the vector of concentrations or
mass fractions of the species depending on the spatial variable $x\in\Omega$
and on time $t>0$, and $\eps>0$ is a characteristic length scale. Furthermore,
$P(y) = {\rm diag}(P_1(y),\ldots,P_d(y))$ is a diagonal matrix, where 
the periodic functions $P_j:Y\to\R$ describe the heterogeneity of the
medium and $Y=(0,b_1)\times\cdots\times(0,b_d)$ with $b_i>0$ is the 
``periodicity cell'',
$a_{ij}:\R^n\to\R$ are the density-dependent diffusion coefficients, 
$f_i:\R^n\to\R$ models the reactions, and $\nu(x)$ is the exterior unit
normal vector to $\pa\Omega$. The theory works also for reaction terms depending 
on $x/\eps$, but we do not consider this dependence to simplify the presentation.
 The divergence operator is understood in the following sense:
$$
  \diver\bigg(\sum_{j=1}^n P\bigg(\frac{x}{\eps}\bigg)a_{ij}(u^\eps)\na u_j^\eps
  \bigg) = \sum_{k=1}^d \frac{\pa}{\pa x_k}\bigg(\sum_{j=1}^n 
	P_k\bigg(\frac{x}{\eps}\bigg)a_{ij}(u^\eps)\frac{\pa u_j^\eps}{\pa x_k}\bigg).
$$
The second homogenization limit is shown in cross-diffusion systems solved 
in a perforated domain. A perforated
domain $\Omega^\eps$ is obtained by removing a subset $\Omega_0^\eps$ 
from $\Omega$, which gives $\Omega^\eps=\Omega\backslash\Omega_0^\eps$. 
The set $\Omega_0^\eps$ may consist of periodically distributed holes 
in the original domain. More precisely,
we introduce the reference set $Y\subset\R^d$ and the set $Y_0\subset Y$ 
(the reference hole)
with Lipschitz boundary $\Gamma=\pa Y_0$, 
satisfying $\overline{Y}_0\subset Y$, and $Y_1 = Y \setminus \overline{Y}_0$.
Then $\Omega_0^\eps$ and the corresponding boundary are defined by
$$
  \Omega_0^\eps = \bigcup_{\xi\in\Xi^\eps}\eps(Y_0+\xi), \quad
	\Gamma^\eps = \bigcup_{\xi\in\Xi^\eps}\eps(\Gamma+\xi),
$$	
where $\Xi^\eps=\{\xi\in\R^d:\eps(\overline{Y}+\xi)\subset\Omega\}$,
and the microscopic model in the perforated domain $\Omega^\eps$ reads as
\begin{equation}\label{2.eq}
  \pa_t u_i^\eps - \diver\bigg(\sum_{j=1}^n a_{ij}(u^\eps)\na u_j^\eps\bigg)
	= f_i(u^\eps) \quad\text{in }\Omega^\eps, \ t>0, \ i=1,\ldots,n,
\end{equation}
together with the boundary and initial conditions
\begin{equation}\label{2.bic}
  \sum_{j=1}^n  a_{ij}(u^\eps) \na u_j^\eps\cdot\nu = 0
	\mbox{ on }\pa\Omega\cup \Gamma^\eps,\ t>0, \quad 
	u_i^\eps(0)=u_i^0\mbox{ in }\Omega^\eps.	
\end{equation}

A key feature of \eqref{1.eq} and \eqref{2.eq} is that the diffusion matrix 
$A(u)=(a_{ij}(u))$ is generally neither symmetric nor positive semi-definite; 
see \cite{Jue15,Jue16} for examples from applications in physics and biology. 
Two examples are presented in Appendix \ref{app.ex}.
To ensure the global existence of weak solutions of problem \eqref{1.eq}-\eqref{1.bic} 
or \eqref{2.eq}-\eqref{2.bic}, we assume that the diffusion system has an
{\em entropy structure}, i.e., there exists a convex function $h\in C^2(\dom;\R)$ with 
$\dom\subset\R^n$ such that the matrix product $h''(u)A(u)$, where $h''(u)$ 
denotes the Hessian of $h$, is positive semi-definite.
Then the so-called entropy $H(u)=\int_\Omega h(u)dx$ is a Lyapunov functional
if $f_i\equiv 0$:
\begin{equation}\label{1.dHdt}
  \frac{dH}{dt} = -\int_\Omega\na u:h''(u)A(u)\na u dx \le 0,
\end{equation}
where ``:'' denotes the Frobenius matrix product.
Gradient estimates, needed for the analysis, are obtained by making a stronger
condition on $h''(u)A(u)$ than just positive semi-definiteness. Since strict
positive definiteness cannot be expected from the applications, we assume that
$h''(u)A(u)$ is ``degenerate'' positive definite. 
We investigate two types of degeneracies, a local and a nonlocal one.

\subsection*{Locally degeneracy structure} 

We assume that $h''(u)A(u)\ge\alpha\, \mbox{diag}((u_i)^{2s_i})_{i=1}^n$ 
in the sense of symmetric matrices
and with $\alpha>0$, $s_i>-1$. Then \eqref{1.dHdt} becomes (still with $f_i\equiv 0$)
$$
  \frac{dH}{dt} + \alpha\sum_{i=1}^n\int_\Omega u_i^{2s_i}|\na u_i|^2 dx \le 0,
$$
leading to $L^2$-estimates for $\na u_i^{s_i+1}$.
Gradient estimates of such a type are well known in the analysis of the porous-medium
equation. The analysis requires a further assumption: The domain $\dom$ is bounded and 
the derivative $h':\dom\to\R^n$ is invertible. Examples are
Boltzmann-type entropies containing expressions like $u_i\log u_i$.
As shown in \cite{Jue15}, this leads to $u_i(x,t)\in\overline{\dom}$ for
$x\in\Omega$, $t>0$, and hence to $L^\infty$-estimates for $u_i$ (without the
use of a maximum principle). 
Using a nonlinear Aubin-Lions lemma, the global existence of bounded weak
solutions was proved in \cite{Jue15} under the condition that the domain $\dom$ is
bounded. Even when $\dom$ is not bounded, the entropy method can be applied,
giving global weak solutions (but possibly not bounded) \cite[Section 4.5]{Jue16}.

\subsection*{Nonlocally degeneracy structure}
 
As an example of a nonlocally degenerate structure, we consider
cross-diffusion systems with coefficients
\begin{equation}\label{1.nonloc}
  a_{ij}(u) = D_i(\delta_{ij}u_{n+1}+u_i), \quad i,j=1,\ldots,n,
\end{equation}
where $\delta_{ij}$ is the Kronecker delta symbol, $u_{n+1}=1-\sum_{i=1}^nu_i$,
and $D_i>0$ for $i=1,\ldots,n$ are diffusion coefficients. 
Such models are used for the transport of ions
through biological channels, where $u_i$ are the ion volume fractions
and $u_{n+1}$ is the solvent concentration. The entropy density is given by
\begin{equation}\label{1.h}
  h(u) = \sum_{i=1}^{n+1}u_i(\log u_i-1)\quad\mbox{for }u=(u_1,\ldots,u_n)\in\dom,
\end{equation}
where $\dom=\{(u_1,\ldots,u_n)\in\R^n:u_1,\ldots,u_n,u_{n+1}>0\}$.
Then $h''(u)A(u)$ is positive semi-definite and if $f_i\equiv 0$,  it holds that
(see \cite[Section 4.6]{Jue16} and \cite[Theorem 1]{ZaJu17})
$$
  \frac{dH}{dt} + \sum_{i=1}^n D_i
	\int_\Omega\big(u_{n+1}|\na u_i^{1/2}|^2 + |\na u_{n+1}|^2\big)dx \le 0.
$$
This gives an $L^2$-estimate for $\na u_{n+1}$, but
generally not for $\na u_i$ because of the factor $u_{n+1}$
which may vanish. We call this
a nonlocal degeneracy since the degeneracy $u_{n+1}$ depends on $u_i$ in
a nonlocal way through the other components $u_j$ for $j\neq i$.

We note that our results can be extended to more general coefficients of the form
$$
  a_{ij}(u) = sD_i u_i^{s-1} q(u_{n+1})\delta_{ij} + D_i u_i^s q'(u_{n+1}),
	\quad i,j=1,\ldots,n,
$$
where $s=1$ or $s=2$ and $q\in C^2([0,1])$ is a positive and nondecreasing function
satisfying $q(0)=0$ and $q'(\xi)\ge\gamma q(\xi)$ for some $\gamma>0$ and
all $\xi\in(0,1)$.

\bigskip
To prove the convergence of solutions of the microscopic problems to a solution of 
the corresponding macroscopic equations, we derive some a priori estimates for 
$(u_i^\eps)$ independent of $\eps$. Compared to \cite{Jue15}, the main novelty
is the derivation of equicontinuous estimates for $(u_i^\eps)$ with respect to 
the time variable. 
This will allow us to obtain compactness properties for a sequence of solutions of 
the microscopic problem defined in a perforated domain. Notice that estimates 
for a discrete time derivative of $(u^\eps)$ in $L^2(0,T; H^1(\Omega^\eps)')$
do not ensure a priori estimates uniform in $\eps$ for the discrete time derivative 
for an extension of $u^\eps$ from $\Omega^\eps$ into $\Omega$. Another important step 
of the analysis presented here is the proof of an existence result for the 
degenerate unit-cell problem, which determines the macroscopic diffusion matrix.   
Here, we apply a regularization technique and use the structure and assumptions 
on the matrix $A(u)$. 

For the first main result on locally degenerate systems, we impose
the following assumptions:
\begin{labeling}{A11}
\item[{\bf A1.}] Entropy: There exists a convex function $h\in C^2(\dom;\R)$
such that $h':\dom\to\R^n$ is invertible, where $\dom\subset(0,1)^n$ is open
and $n\ge 1$.

\item[{\bf A2.}] ``Degenerate'' positive definiteness: There exist numbers $s_i > -1$
($i=1,\ldots,n$) and $\alpha>0$
such that for $z=(z_1,\ldots,z_n)\in\R^n$, $u=(u_1,\ldots,u_n)\in\dom$,
$$
  z^\top  h^{\prime\prime}(u)A(u)z \ge \alpha\sum_{i=1}^n |u_i|^{2s_i} z_i^2. 
$$
\item[{\bf A3.}] Diffusion coefficients: Let  
$A(u)=(a_{ij}(u))\in C^0(\dom;\R^{n\times n})$. 
There exists a constant $C_A>0$ such that for all $u\in\dom$ and for those 
$j=1,\ldots,n$ such that $s_j>0$, it holds that 
$$
  |a_{ij}(u)|\le C_A u_j^{s_j} \quad\mbox{for }i=1,\ldots,n.
$$
Furthermore, $P\in L^\infty(Y;\R^{d\times d})$ with 
$P(y) ={\rm diag}(P_1(y),\ldots,P_d(y))$ satisfies $P_i(y)\ge d_0>0$ in
$Y$ for some $d_0>0$ and for all $i=1, \ldots, d$.

\item[{\bf A4.}] Reaction terms: $f\in C^0(\overline\dom;\R^n)$ 
and there exists $C_f>0$ such that $f(u)\cdot h'(u)\le C_f(1+h(u))$ for $u\in\dom$.

\item[{\bf A5.}] Initial datum:  $u^0:\Omega\to\R^n$ is measurable and 
$u^0(x)\in\dom$ for $x\in\Omega$.

\item[{\bf A6.}] Bound for the matrix $h''(u)A(u)$: There exists a constant
$C>0$ such that for all $u\in\dom$ and $i,j=1,\ldots,n$,
$$
  (h''(u)A(u))_{ij} \le Cu_i^{s_i}u_j^{s_j}.
$$
\end{labeling}

Let us discuss these assumptions.
As mentioned above, Assumption A1 guarantees the $L^\infty$ boundedness of
the solutions. 
Assumption A2 is needed for the compactness argument. For the existence analysis,
it can be weakened to continuous functions instead of power-law functions
\cite{Mou16},  but the convergence $\eps\to 0$ is more delicate.
The growth estimate
for $a_{ij}(u)$ in Assumption~A3 is crucial for  the proof of the 
equicontinuity property with respect to the time variable. 
The growth condition on $f_i$ in Assumption A4
allows us to handle the reaction terms. The latter condition generally rules out
quadratic growth of the concentrations; we refer to
\cite{FLS16} for reaction-diffusion systems with diagonal diffusion matrices but
quadratic reaction terms. Assumption A5 guarantees that the initial datum is
bounded; it can be relaxed to $u^0(x)\in\overline{\dom}$.
Finally, Assumption A6 is a technical condition
to ensure the solvability of the unit-cell problems. 
In Appendix \ref{app.ex}, we give two examples
from applications, for which the assumptions are satisfied.

To simplify the presentation, we introduce some notation:
$$
  P_k^\eps(x) = P_k(x/\eps)\mbox{ for }x\in\Omega,\ k = 1, \ldots, d, \quad
	\Omega_T = \Omega\times(0,T), \quad \Omega^\eps_T = \Omega^\eps\times(0,T).
$$

\begin{definition} 
A weak solution of problem \eqref{1.eq}-\eqref{1.bic} is a function  $u^\eps \in 
L^\infty(0, T; L^\infty(\Omega;\R^n))$ with 
$(u^\eps_i)^{s_i+1} \in L^2(0,T; H^1(\Omega))$ and $\partial_t u^\eps_i 
\in L^2(0,T; H^1(\Omega)')$  for $i=1, \ldots, n$, satisfying 
$$
   \int_0^T\sum_{i=1}^n\langle\pa_t u_i^\eps,\varphi_i\rangle  dt
	+ \int_0^T\int_\Omega \bigg(\sum_{i,j=1}^n P^\eps(x)a_{ij}(u^\eps)
	\na u_j^\eps\cdot\na \varphi_i- \sum_{i=1}^n f_i(u^\eps)\varphi_i \bigg)dxdt = 0,  
$$
for all $\varphi \in L^2(0, T; H^1(\Omega;\R^n))$, and the initial conditions 
are satisfied in the $L^2$ sense. 

A weak solution of problem \eqref{2.eq}-\eqref{2.bic} 
is defined in a similar way by replacing $\Omega$ by $\Omega^\eps$. 
\end{definition} 
Here, $\langle \psi, \varphi \rangle$ denotes the dual product between 
$\psi \in H^1(\Omega)'$ and $\varphi \in H^1(\Omega)$  and the expression
$P^\eps\na u_j^\eps\cdot\na\varphi_i$ is the sum
$\sum_{k=1}^d P_k^\eps\pa_{x_k}u_j^\eps \pa_{x_k}\varphi_i$.

\begin{theorem}[Homogenization limit for problems with local degeneracy]\label{thm1}
Let Assumptions A1-A6 hold.

{\rm (i)} Let $u^\eps$ be a weak solution of the microscopic system
\eqref{1.eq}-\eqref{1.bic}. 
Then there exists a subsequence of $(u^\eps)$, which is not relabeled, 
such that $u^\eps\to u$ strongly in $L^p(\Omega_T;\R^n)$ 
for all $p<\infty$  as $\eps\to 0$, and the limit function 
$u \in L^\infty(0, T; L^\infty(\Omega;\R^n))$, with 
$u_i^{s_i+1}\in L^2(0,T;H^1(\Omega))$ and    
$\pa_t u_i \in  L^2(0,T; H^1(\Omega)')$ for $i=1,\ldots, n$,
solves the macroscopic system
\begin{equation}\label{macro_1}
\begin{aligned}
  & \pa_t u_i - \sum_{k,m=1}^d\sum_{\ell=1}^n \frac{\pa}{\pa x_m}\bigg(
	B^{i\ell}_{mk}(u)\frac{\pa u_\ell}{\pa x_k}\bigg) = f_i(u) \quad\mbox{in }
	\Omega,\ t>0, \ i=1,\ldots,n, \\
  & \sum_{k,m=1}^d\sum_{\ell=1}^n \nu_m
	B^{i\ell}_{mk}(u)\frac{\pa u_\ell}{\pa x_k} = 0\mbox{ on }\pa\Omega,
	\ t>0, \quad u_i(0)=u^0_i\mbox{ in }\Omega,
\end{aligned}
\end{equation}
where $(B^{i\ell}_{mk}(u))$ is the homogenized diffusion matrix 
defined in \eqref{B}.

{\rm (ii)} Let  $u^\eps$ be a weak solution of the microscopic system 
\eqref{2.eq}-\eqref{2.bic}. Then, up to a subsequence and by identifying 
$u^\eps$ with its extension from $\Omega^\eps$ into $\Omega$,  $u^\eps\to u$ 
strongly in $L^p(\Omega_T;\R^n)$ for $p<\infty$, 
where $u$, with $u_i^{s_i+1}\in L^2(0,T;H^1(\Omega))$   
and $\pa_t u_i \in  L^2(0,T; H^1(\Omega)')$ for $i=1,\ldots,n$, 
is a solution of \eqref{macro_1} 
with the macroscopic diffusion matrix $(B^{i\ell}_{mk}(u))$  defined in \eqref{B_2}. 
\end{theorem}

For nonlocally degenerate systems \eqref{1.eq} or \eqref{2.eq}
with diffusion coefficients \eqref{1.nonloc},
the weak solution is defined in a slightly different way than usually, since
the regularity $u_i^\eps\in L^2(0,T;H^1(\Omega))$ may not hold.
We recall the definition from \cite{Jue15}.

\begin{definition}\label{def.weak}
A weak solution of \eqref{1.eq}-\eqref{1.bic} with diffusion coefficients
\eqref{1.nonloc} are functions $u_1^\eps,\ldots,u_n^\eps$ and
$u_{n+1}^\eps=1-\sum_{i=1}^n u_i^\eps$ satisfying $u_i^\eps\ge 0$, $u_{n+1}^\eps\ge 0$
in $\Omega_T$, $u_i^\eps\in L^\infty(0,T;L^\infty(\Omega))$, 
$(u_{n+1}^\eps)^{1/2}$, $(u_{n+1}^\eps )^{1/2}u_i^\eps\in
L^2(0,T;H^1(\Omega))$,
$\pa_t u_i^\eps \in  L^2(0,T; H^1(\Omega)')$ for $i=1,\ldots,n$, and
\begin{equation}\label{weak_2}
\begin{aligned}
  \int_0^T\sum_{i=1}^n \langle \pa_t u_i,\varphi_i\rangle dt 
  + &\int_0^T\int_\Omega\sum_{i=1}^n P^\eps(x)D_i(u_{n+1}^\eps)^{1/2} \\
	&{}\times
	\Big(\na\big((u_{n+1}^\eps)^{1/2}u_i^\eps\big) - 3u_i^\eps\na(u_{n+1}^\eps)^{1/2}
	\Big)\cdot\na\varphi_i dxdt = 0
\end{aligned}
\end{equation}
for all $\varphi \in L^2(0, T; H^1(\Omega;\R^n))$, 
and the initial conditions are satisfied in the $H^1(\Omega)'$ sense. 

A weak solution of problem \eqref{2.eq}-\eqref{2.bic} 
with diffusion coefficients \eqref{1.nonloc} is
defined analogously by replacing $\Omega$ by $\Omega^\eps$. 
\end{definition}

\begin{theorem}[Homogenization limit for problems with nonlocal degeneracy]
\label{main_filling}
Let Assumptions A1 and A5 hold. 

{\rm (i)} A subsequence $(u^\eps)$ of solutions of the microscopic problem 
\eqref{1.eq}-\eqref{1.bic}, with the matrix $A$ defined in \eqref{1.nonloc}, 
converges to a solution $u\in L^\infty(0,T;L^\infty(\Omega;\R^n))$, with 
$u_{n+1}^{1/2}$, $u_{n+1}^{1/2}u_i \in
L^2(0,T;H^1(\Omega))$, $\pa_t u_i \in L^2(0,T; H^1(\Omega)')$ for 
$i=1,\ldots,n$, of the macroscopic equations 
\begin{equation}\label{macro_2}
\begin{aligned}
  & \partial_t u - \diver\big(D_{\rm hom}A(u) \na u\big) = 0 && \text{ in }\Omega, 
	\ t>0, \\  
  & D_{\rm hom}A(u) \na u \cdot \nu = 0 && \text{ on } \pa\Omega, \ t>0, \quad
 u(0) = u_0 \text{ in }\Omega,
\end{aligned}
\end{equation}
where the macroscopic matrix $D_{\rm hom}$ is defined in \eqref{Dhom}.

{\rm (ii)} In the case of the microscopic problem \eqref{2.eq}-\eqref{2.bic}, 
we obtain the same macroscopic equations as in \eqref{macro_2} with a different 
macroscopic diffusion matrix given by \eqref{Dhom_perfor}. 
\end{theorem}


\section{Proof of Theorem \ref{thm1}}\label{sec.thm1}

For the proof the theorem, 
we show some a priori estimates uniform in $\eps$ for solutions of the
microscopic problems \eqref{1.eq}-\eqref{1.bic} and \eqref{2.eq}-\eqref{2.bic}. 
We suppose throughout the section that Assumptions A1-A6 hold.
First, we recall the following elementary inequalities.

\begin{lemma}[H\"older-type inequalities]\label{lem.ineq}
Let $a$, $b\ge 0$ and $p\ge 1$. Then
$$
  |a-b|^p \le |a^p-b^p| \le p(a^{p-1}+b^{p-1})|a-b|.
$$
\end{lemma}

The a priori estimates for problem \eqref{1.eq}-\eqref{1.bic} are as follows.

\begin{lemma}[A priori estimates]\label{lem.est}
For any $\eps>0$, there exists a bounded weak solution $u^\eps$ of problem
\eqref{1.eq}-\eqref{1.bic} such that $u^\eps(x,t)\in \overline{\dom}$ for $x\in\Omega$, 
$t>0$ and
\begin{align}
  \|(u_i^\eps)^{s_i+1}\|_{L^2(0,T;H^1(\Omega))}
	&\le C &&\quad\mbox{for } i=1, \ldots, n, \label{est1} \\
	\|u_i^\eps\|_{L^2(0,T;H^1(\Omega))} &\le C 
	&&\quad\mbox{for } -1<s_i\le 0, \label{est2} \\
	\|\vartheta_\tau u_i^\eps - u_i^\eps\|_{L^2((0,T-\tau)\times\Omega)}
	&\le C\tau^{1/4} &&\quad\mbox{for } -1<s_i\le 0, \label{est3} \\
	\|\vartheta_\tau u_i^\eps - u_i^\eps\|_{L^{2+s_i}((0,T-\tau)\times\Omega)}
	&\le C\tau^{1/(4+2s_i)} &&\quad\mbox{for }s_i>0, \label{est4}
\end{align}
where $\vartheta_\tau u_i^\eps(x,t)=u_i^\eps(x,t+\tau)$ for $x\in\Omega$ and
$t\in(0,T-\tau)$, for $\tau \in (0, T)$, 
and the constant $C>0$ is independent of $\eps$.
\end{lemma}

\begin{proof}
Theorem 2 in \cite{Jue15} shows that there exists a bounded weak solution 
$u^\eps$ to \eqref{1.eq}-\eqref{1.bic} satisfying $u^\eps(x,t)\in\overline{\dom}$
for $x\in\Omega$, $t>0$. Estimates \eqref{est1}-\eqref{est2}
are a consequence of the entropy production inequality, which is obtained by
taking an approximation of $(\pa h/\pa u_i)(u^\eps)$ as a test function in \eqref{1.eq}. 
Notice that the dependence on $x\in \Omega$ is via multiplication by a 
diagonal matrix $P^\eps(x)$, so the entropy $h(u)$ does not depend explicitly on $x$.   
Since the entropy $h$ is generally undefined on
$\pa\dom$, the equations in \cite{Jue15} have been approximated, and the
existence of a family of approximate solutions satisfying 
\eqref{est1} has been proved. Then the convergence of the approximate solutions 
in appropriate spaces for vanishing approximation 
parameters directly leads to \eqref{est1}.  
Thanks to the positive lower bound for $P$ (uniform in $\eps$), we see that estimate
\eqref{est1} is independent of $\eps$.

Estimate \eqref{est2} for $-1<s_i \leq  0$ follows from \eqref{est1} and the
boundedness of $u^\eps$:
$$
  \|\na u_i^\eps\|_{L^2(\Omega_T)} 
	= \frac{1}{s_i+1}\|u_i^\eps\|_{L^\infty(\Omega_T)}^{-s_i}
	\|\na(u_i^\eps)^{s_i+1}\|_{L^2(\Omega_T)} \le C,
$$
for $i=1, \ldots, n$, where $C>0$ is here and in
the following a generic constant independent of $\eps$.
The boundedness of $(u^\eps)$ (uniform in $\eps$) is ensured by the assumptions on 
$h$, see Assumption A1. 

It remains to show \eqref{est3} and \eqref{est4}. For this, we use the 
(admissible) test function $\phi=(\phi_1,\ldots,\phi_n)$ with
\begin{align*}
  \phi_i(x,t) &= \int_{t-\tau}^t(\vartheta_\tau u_i^\eps(x,\sigma)-u_i^\eps(x,\sigma))
	\kappa(\sigma)d\sigma 
	&& \mbox{if }s_i\le 0, \\
  \phi_i(x,t) &= \int_{t-\tau}^t\big((\vartheta_\tau u_i^\eps(x,\sigma))^{s_i+1}
	- (u_i^\eps(x,\sigma))^{s_i+1}\big)\kappa(\sigma)d\sigma &&\mbox{if }s_i>0,
\end{align*}
where $\tau\in(0,T)$, $i=1,\ldots,n$, 
$\kappa(\sigma)=1$ for $\sigma\in(0,T-\tau)$ and $\kappa(\sigma)=0$
for $\sigma\in[-\tau,0]\cup[T-\tau,T]$. This gives
\begin{align*}
  0 &= \int_0^T\sum_{i=1}^n\langle\pa_t u_i^\eps,\phi_i\rangle dt
	+ \int_0^T\int_\Omega \sum_{i,j=1}^n P^\eps(x)a_{ij}(u^\eps)
	\na u_j^\eps\cdot\na \phi_i dxdt \\
	&\phantom{xx}{}- \int_0^T\int_\Omega\sum_{i=1}^n f_i(u^\eps)\phi_i dxdt 
	=: I_1 + I_2 + I_3.
\end{align*}
We integrate by parts in the first integral, taking into account that 
$\phi_i(0)=\phi_i(T)=0$. Then, for all $i=1, \ldots, n$ such that $s_i\le 0$,
\begin{align*}
  \int_0^T\langle\pa_t u_i^\eps,\phi_i\rangle dt
	&= -\int_0^T\int_\Omega u_i^\eps\pa_t\phi_i dxdt \\
	&= -\int_0^{T-\tau}\int_\Omega
	u_i^\eps(\vartheta_\tau u_i^\eps-u_i^\eps)dxdt
  + \int_\tau^T\int_\Omega
	u_i^\eps(u_i^\eps-\vartheta_{-\tau} u_i^\eps)dxdt \\
	&= -\int_0^{T-\tau}\int_\Omega
	u_i^\eps(\vartheta_\tau u_i^\eps-u_i^\eps)dxdt
  + \int_0^{T-\tau}\int_\Omega
	\vartheta_\tau u_i^\eps(\vartheta_\tau u_i^\eps	- u_i^\eps)dxdt \\
	&= \int_0^{T-\tau}\int_\Omega(\vartheta_\tau u_i^\eps-u_i^\eps)^2 dxdt.
\end{align*}
In a similar way, for those $i=1, \ldots, n$ such that  $s_i>0$, 
$$
  \int_0^T\langle\pa_t u_i^\eps,\phi_i\rangle dt
	= \int_0^{T-\tau}\int_\Omega(\vartheta_\tau u_i^\eps-u_i^\eps)
	\big((\vartheta_\tau u_i^\eps)^{s_i+1}- (u_i^\eps)^{s_i+1} \big) dxdt.
$$
Lemma \ref{lem.ineq} with $p=s_i+1$ gives
$$
  |\vartheta_\tau u_i^\eps-u_i^\eps|^{s_i+1}
  \le \big|(\vartheta_\tau u_i^\eps)^{s_i+1} -(u_i^\eps)^{s_i+1}\big|.
$$
Thus, still in the case $s_i>0$,
$$ 
  \int_0^T\langle\pa_t u_i^\eps,\phi_i\rangle dt 
	\ge \int_0^{T-\tau}\int_\Omega
	(\vartheta_\tau u_i^\eps-u_i^\eps)^{s_i+2} dxdt.
$$
We conclude that
$$
  I_1 \ge \int_0^{T-\tau}\int_\Omega\bigg(\sum_{i=1,\,s_i>0}^n
	(\vartheta_\tau u_i^\eps-u_i^\eps)^{s_i+2} 
	+ \sum_{i=1,\,s_i\le 0}^n(\vartheta_\tau u_i^\eps-u_i^\eps)^2\bigg) dxdt.
$$

For the second integral $I_2$, we use the relation
$$
   \int_0^T w(t)\int_{t-\tau}^t v(\sigma)\, d\sigma dt 
	= \int_0^{T-\tau}\int_t^{t+\tau}w(\sigma)d\sigma\,  v(t) dt, 
$$
where $v(t) = 0$ for $t \in [-\tau, 0]\cup[T-\tau, T]$,  to infer that
\begin{align*}
  I_2 &= \int_{\Omega_{T-\tau}} \sum_{i,j=1,\, s_i >0}^n P^\eps(x)\na
	\big((\vartheta_\tau u_i^\eps)^{s_i+1}-(u_i^\eps)^{s_i+1}\big)
	\cdot\int_t^{t+\tau} a_{ij}(u^\eps)\na u_j^\eps d\sigma dxdt \\
	&\phantom{xx}{}+ \int_{\Omega_{T-\tau}} \sum_{i,j=1,\, s_i \leq 0}^n P^\eps(x)\na
	(\vartheta_\tau u_i^\eps- u_i^\eps)	\cdot\int_t^{t+\tau} 
	a_{ij}(u^\eps)\na u_j^\eps d\sigma dxdt.
\end{align*}
Again, we distinguish between the cases $s_i\le 0$ and $s_i>0$. Employing the
Cauchy-Schwarz inequality  we have
\begin{equation*}
\begin{aligned}
  |I_2| &\le \tau^{1/2}C\sum_{\substack{i,j=1,\, s_i, s_j>0}}^n
	\bigg\|\frac{a_{ij}(u^\eps)}{(u_j^\eps)^{s_j}}\bigg\|_{L^\infty(\Omega_T)}
	\|\na (u_i^\eps)^{s_i+1}\|_{L^2(\Omega_T)}
	\|\na (u_j^\eps)^{s_j+1}\|_{L^2(\Omega_T)} \\
	&\phantom{xx}{}+ \tau^{1/2}C\sum_{\substack{i,j=1,\, s_i>0,\,  s_j\le 0}}^n
	\|a_{ij}(u^\eps)\|_{L^\infty(\Omega_T)}
	\|\na(u_i^\eps)^{s_i+1}\|_{L^2(\Omega_T)}
	\|\na u_j^\eps\|_{L^2(\Omega_T)}  \\
	&\phantom{xx}{} +\tau^{1/2}C\sum_{\substack{i,j=1,\, s_i \leq 0,\, s_j>0}}^n
	\bigg\|\frac{a_{ij}(u^\eps)}{(u_j^\eps)^{s_j}}\bigg\|_{L^\infty(\Omega_T)}
	\|\na  u_i^\eps\|_{L^2(\Omega_T)}
	\|\na (u_j^\eps)^{s_j+1}\|_{L^2(\Omega_T)} \\
	&\phantom{xx}{}+ \tau^{1/2}C\sum_{\substack{i,j=1,\, s_i,s_j\le 0}}^n
	\|a_{ij}(u^\eps)\|_{L^\infty(\Omega_T)}
	\|\na u_i^\eps\|_{L^2(\Omega_T)}
	\|\na u_j^\eps\|_{L^2(\Omega_T)} 
	\le  C\tau^{1/2},
	\end{aligned}
\end{equation*}
in view of Assumption~A3 and estimates \eqref{est1}-\eqref{est2}.

It remains to estimate $I_3$. The boundedness of $u^\eps$ yields
\begin{align*}
 | I_3| &\leq  \int_0^{T-\tau}\int_\Omega\sum_{i=1,\,s_i>0}^n\int_t^{t+\tau}
	|f_i(u^\eps)|ds \,  \big|(\vartheta_\tau u_i^\eps)^{s_i+1}
	- (u_i^\eps)^{s_i+1}\big|  dxdt \\
	&\phantom{xx}{} +  \int_0^{T-\tau}\int_\Omega\sum_{i=1,\,s_i\le 0}^n\int_t^{t+\tau}
	| f_i(u^\eps)|ds \, |\vartheta_\tau u_i^\eps - u_i^\eps|\,  dxdt   \le C\tau.
\end{align*}
Putting these estimates together, we infer that \eqref{est3} for $s_i\le 0$
and \eqref{est4} for $s_i>0$ holds, concluding the proof.
\end{proof}

\begin{lemma}[A priori estimates] \label{lem.est_2}
For any $\eps>0$, there exists a bounded weak solution $u^\eps$ to 
\eqref{2.eq}-\eqref{2.bic} such that $u^\eps(x,t)\in \overline\dom$ for $x\in\Omega$, 
$t>0$ and
\begin{align}
  \|(u_i^\eps)^{s_i+1}\|_{L^2(0,T;H^1(\Omega^\eps))}
	&\le C &&\mbox{for } i=1,\ldots,n, \label{est21} \\
	\|u_i^\eps\|_{L^2(0,T;H^1(\Omega^\eps))} 
	&\le C &&\mbox{for } -1<s_i\le 0, \label{est22} \\
	\|\vartheta_\tau u_i^\eps - u_i^\eps\|_{L^2((0,T-\tau)\times \Omega^\eps)}
	&\le C\tau^{1/4} &&\mbox{for } -1<s_i\le 0, \label{est23} \\
	\|\vartheta_\tau u_i^\eps - u_i^\eps\|_{L^{2+s_i}((0,T-\tau)\times\Omega^\eps)}
	&\le C\tau^{1/(4+2s_i)} &&\mbox{for } s_i>0, \label{est24}
\end{align}
where $\vartheta_\tau u_i^\eps(x,t)=u_i^\eps(x,t+\tau)$ for $x\in\Omega^\eps$,
$t\in(0,T-\tau)$, and the constant $C>0$ is independent of $\eps$.
\end{lemma}

\begin{proof}
The proof of a priori estimates \eqref{est21}-\eqref{est24} follows the same steps 
as in the proof of Lemma \ref{lem.est}. Thanks to the structure of the proof, 
all estimates in Lemma \ref{lem.est} can be obtained for $\Omega^\eps$ instead of 
$\Omega$, independently of $\eps$. 
\end{proof}

\begin{remark}[Extension]\rm\label{extension} 
Our assumptions on the microscopic structure of $\Omega^\eps$ ensure that there 
exists an extension $\overline{u^\eps_i}$ of $u^\eps_i$ and 
$\overline{(u^\eps_i)^{s_i + 1}}$ of $(u^\eps_i)^{s_i +1}$ from 
$\Omega^\eps$ to $\Omega$ with the properties 
\begin{align*} 
  & \|\overline{u^\eps_i} \|_{L^2(\Omega)} 
	\leq \mu \| u^\eps_i \|_{L^2(\Omega^\eps)}, \quad 
  \|\nabla  \overline{u^\eps_i} \|_{L^2(\Omega)} 
	\leq \mu \|\nabla  u^\eps_i\|_{L^2(\Omega^\eps)}  \text{ for } -1< s_i\leq 0, \\
  & \|\overline{(u^\eps_i)^{s_i + 1}}\|_{L^2(\Omega)} 
	\leq \mu \|(u^\eps_i)^{s_i+1} \|_{L^2(\Omega^\eps)}, 
  \quad \|\nabla \overline{(u^\eps_i)^{s_i + 1}}\|_{L^2(\Omega)} 
	\leq \mu \|\nabla (u^\eps_i)^{s_i+1}\|_{L^2(\Omega^\eps)}
\end{align*} 
for $t>0$, where $\mu>0$ is some constant independent of $\eps$; see, e.g., 
\cite{Cioranescu_book} or Appendix \ref{app} for details. 
\qed
\end{remark}

\begin{lemma}[Convergence]\label{conver_11}
Let $u^\eps$ be a weak solution of
\eqref{1.eq}-\eqref{1.bic} or \eqref{2.eq}-\eqref{2.bic}.
Then there exists a subsequence of $(u^\eps)$, which is not relabeled, 
and functions $u\in L^\infty(0,T; L^\infty(\Omega; \R^n))$, with $u_i^{s_i+1} 
\in L^2(0,T;H^1(\Omega))$ for $i=1,\ldots, n$, 
$V_1,\ldots,V_n\in L^2(\Omega_T;H^1_{\rm per} (Y)/\R)$ such that, as $\eps \to 0$,
\begin{align}
  u^\eps_i &\to u_i &&\quad\mbox{strongly in }L^p(\Omega_T),\ p<\infty, \label{conv1} \\
  (u_i^\eps)^{s_i+1} &\to u_i^{s_i+1}
	&&\quad\mbox{strongly in }L^2(\Omega_T) \text{ for } s_i >0, \label{conv2} \\
  \na (u_i^\eps)^{s_i+1} &\rightharpoonup \na u_i^{s_i+1}
	&&\quad\mbox{weakly in }L^2(\Omega_T), \label{conv3} \\
  \na (u_i^\eps)^{s_i+1} &\rightharpoonup \na u_i^{s_i+1}+\na_y V_i
	&& \quad\mbox{two-scale}, \ i=1,\ldots,n. \label{conv4} 
\end{align}
In the case of solutions $(u^\eps)$ of \eqref{2.eq}-\eqref{2.bic},
convergence results \eqref{conv1}-\eqref{conv4} hold for a subsequence of 
the  extension of $(u^\eps_i)^{s_i+1}$  and of $u^\eps_i$ from 
$\Omega^\eps_T$ into $\Omega_T$, for $i=1, \ldots, n$, 
considered in Remark \ref{extension}.  
\end{lemma}

\begin{proof}
For $s_i\le 0$, estimates \eqref{est2} and \eqref{est3} allow us to apply 
the Aubin-Lions lemma in the version of \cite{Sim87}, giving the existence 
of a subsequence, not relabeled, such that
$u_i^\eps\to u_i$ strongly in $L^{2}(\Omega_T)$. Since $(u_i^\eps)$ is
bounded in $L^\infty(\Omega_T)$, by construction, this convergence even holds
in $L^p(\Omega_T)$ for any $p<\infty$. 

For $s_i>0$, we apply Lemma \ref{lem.ineq} with $p=s_i+1$ and the bounds
$|\vartheta_\tau u_i^\eps|\le 1$, $|u_i^\eps|\le 1$:
$$
  \|\vartheta_\tau  (u_i^\eps)^{s_i+1} 
	- (u_i^\eps)^{s_i+1}\|_{L^{2+s_i}(\Omega_{T-\tau})}
	\leq C \|\vartheta_\tau u_i^\eps - u_i^\eps\|_{L^{2+s_i}(\Omega_{T-\tau})}
	\leq C \tau^{1/(4+2s_i)}.
$$
Hence, by applying the Aubin-Lions lemma of \cite{Sim87} to 
$(u_i^\eps)^{s_i+1}$, we deduce the strong convergence 
$(u_i^\eps)^{s_i+1}\to w_i$ in $L^2(\Omega_T)$  as $\eps \to 0$
for some $w_i\in L^2(\Omega_T)$
with $w_i\ge 0$. In particular, up to a subsequence, we have
$(u_i^\eps)^{s_i+1}\to w_i$ a.e.\ in $\Omega_T$ and consequently,
$u_i^\eps\to u_i:=w_i^{1/(s_i+1)}$ a.e.\ in $\Omega_T$. 
Since $(u_i^\eps)$ is bounded in $L^\infty(\Omega_T)$, it follows that 
$u_i^\eps\to u_i$ strongly in $L^p(\Omega_T)$ for any $p<\infty$ and also
$(u_i^\eps)^{s_i+1}\to (u_i)^{s_i+1}$ strongly in
$L^2(\Omega_T)$, which proves \eqref{conv2}.

Convergence \eqref{conv3} follows from the bound
\eqref{est1}, possibly after extracting another subsequence.
Finally, the two-scale convergence \eqref{conv4} is a
consequence of the boundedness of $\na (u_i^\eps)^{s_i+1}$ 
in $L^2(\Omega_T)$; see, e.g., \cite{All92, Ngu89} or Lemma \ref{lem1} 
in Appendix \ref{app}.

In the case of solutions $(u^\eps)$ of problem \eqref{2.eq}-\eqref{2.bic}, 
we consider extensions of $u^\eps_i$ and $(u^\eps_i)^{s_i+1}$ from $\Omega^\eps$  
into $\Omega$ as in Remark \ref{extension}, for $i=1, \ldots, n$.  
The properties and the linearity of the extension and the a priori estimates from 
Lemma \ref{lem.est_2} imply the corresponding estimates for 
$\overline{u^\eps_i}$, for those $i$ such that $-1<s_i \leq 0$,  
and $\overline{(u^\eps_i)^{s_i + 1}}$ in 
$L^2(0,T; H^1(\Omega))$, for $\vartheta_\tau \overline{u^\eps_i} - \overline{u^\eps_i}$  
in $L^2(\Omega_T)$ if $-1<s_i\le 0$,  
and for $\vartheta_\tau  \overline{(u^\eps_i)^{s_i + 1}} 
- \overline{(u^\eps_i)^{s_i+1}}$ in $L^{2+s_i}(\Omega_T)$ for $s_i >0$. 

We conclude from the estimates for $\overline{u^\eps_i}$ that there exists 
$u_i \in L^2(0, T; H^1(\Omega))$ such that, up to a subsequence,   
$\overline{u^\eps_i}\to u_i$ strongly in $L^2(\Omega_T)$ for $-1< s_i\leq 0$.
Furthermore, the estimates for $\overline{(u^\eps_i)^{s_i + 1}}$ ensure that there 
exists $w \in L^2(0,T; H^1(\Omega;\R^n))$ such that, up to a subsequence,   
$w^\eps_i := \overline{(u^\eps_i)^{s_i + 1}}$ converges strongly to 
$w_i$ in $L^2(\Omega_T)$.  
This ensures also the a.e.\ pointwise convergence of 
$(w^\eps_i)^{1/(s_i+1)} = (\overline{(u^\eps_i)^{s_i + 1}})^{1/(s_i+1)}$ 
to $w_i^{1/(s_i+1)}$ in $\Omega_T$ as $\eps \to 0$.
It remains to prove that $w_i^{1/(s_i+1)}=u_i$.

The properties of the extension imply that $(w^\eps_i)^{1/(s_i+1)}$ 
is uniformly bounded as $u^\eps_i$ is uniformly bounded.
We deduce that, up to a subsequence, $(w^\eps_i)^{1/(s_i+1)} \to w_i^{1/(s_i+1)}$ 
strongly in $L^2(\Omega_T)$. 
Notice that, due to the construction of the extension, we have 
$(\overline{(u^\eps_i)^{s_i + 1}})^{1/(s_i+1)} = u^\eps_i$ in 
$\Omega^\eps_T$. For $-1<s_i\le 0$, we apply Lemma \ref{lem.ineq} to
$a=(u_i^{\eps_m})^{s_i+1}$, $b=(u_i^{\eps_\ell})^{s_i+1}$ and $p=1/(s_i+1)\ge 1$
and use the properties of the extension:
\begin{align*} 
  \big\|\overline{(u^{\eps_m}_i)^{s_i+1}} - \overline{(u^{\eps_\ell}_i)^{s_i+1}} 
	\big\|_{L^2(\Omega_T)}^2  
	&\leq \mu \big\|{(u^{\eps_m}_i)^{s_i+1}} - {(u^{\eps_\ell}_i)^{s_i+1}} 
	\big\|_{L^2(\Omega^\eps_T)}^2 \\ 
	&\leq \mu_1 \|u^{\eps_m}_i  - u^{\eps_\ell}_i \|^{2(s_i+1)}_{L^2(\Omega^\eps_T)}  
	\leq \mu_2 \big\|\overline{u^{\eps_m}_i}  - \overline{u^{\eps_\ell}_i} 
	\big\|^{2(s_i+1)}_{L^2(\Omega_T)}, 
\end{align*}
whereas for $s_i>0$ we obtain 
\begin{align*} 
  \big\|\overline{u^{\eps_m}_i}  - \overline{u^{\eps_\ell}_i} \big\|_{L^2(\Omega_T)}^2
  &\leq \mu \|u^{\eps_m}_i  -  u^{\eps_\ell}_i \|_{L^2(\Omega^\eps_T)}^2 
  \leq \mu_1 \|(u^{\eps_m}_i)^{s_i + 1} - (u^{\eps_\ell}_i)^{s_i+1} 
	\|^{2/(s_i+1)}_{L^2(\Omega^\eps_T)} \\
  &\leq \mu_2 \big\|\overline{(u^{\eps_m}_i)^{s_i+1}} 
	- \overline{(u^{\eps_\ell}_i)^{s_i+1}} \big\|^{2/(s_i+1)}_{L^2(\Omega_T)}. 
\end{align*}
Hence, if $-1< s_i \leq 0$, the strong convergence of $\overline{u^\eps_i}$ 
implies the strong convergence of $\overline{(u^{\eps}_i)^{s_i+1}}$, while 
for $s_i>0$, strong convergence of $\overline{(u^{\eps}_i)^{s_i+1}}$ ensures 
the strong convergence of $\overline{u^\eps_i}$.  Therefore, denoting by
$\chi_{\Omega^\eps}$ the characteristic function of $\Omega^\eps$,
\begin{align*} 
  \frac{|Y_1|}{|Y|} \int_{\Omega_T} u_i \phi dx dt 
  &= \lim_{\eps \to 0} \int_{\Omega_T} \overline{u^\eps_i} \chi_{\Omega^\eps}\phi dx dt 
  = \lim_{\eps \to 0} \int_0^T\int_{\Omega^\eps} u^\eps_i \phi  dx dt  \\
	&=  \lim_{\eps \to 0} \int_{\Omega_T} (w^\eps_i)^{1/(s_i+1)} 
	\chi_{\Omega^\eps} \phi dx dt 
  = \frac{|Y_1|}{|Y|}\int_{\Omega_T}w_i^{1/(s_i+1)} \phi dx dt,  
\end{align*}
for any $\phi \in C^\infty_0(\Omega_T)$.  We deduce that
$w_i^{1/(s_i+1)} = u_i$ and consequently 
$w_i = u_i^{s_i+1}$ a.e.\ in $\Omega_T$. The boundedness of $\na(u^\eps_i)^{s_i+1}$ 
and the properties of the extension ensure the convergence results in \eqref{conv3} 
and \eqref{conv4}. 
Hence, we obtain convergence results \eqref{conv1}-\eqref{conv4}  for a 
subsequence of $(\overline{u^\eps_i})$  and $(\overline{(u^\eps_i)^{s_i + 1}})$, 
respectively, finishing the proof.
\end{proof}

\begin{proof}[Proof of Theorem \ref{thm1}]
Using the convergence results in Lemma \ref{conver_11}, we are now able to derive the
macroscopic equations for microscopic problems \eqref{1.eq}-\eqref{1.bic} 
and \eqref{2.eq}-\eqref{2.bic}. 

The strong convergence of $(u^\eps)$ in $L^p(\Omega_T)$ for any $p<\infty$ 
and Assumption A3 imply that, for those $j$ satisfying $s_j>0$,
$$
  \frac{a_{ij}(u^\eps)}{(u_j^\eps)^{s_j}} \to \frac{a_{ij}(u)}{(u_j)^{s_j}}
	\quad\mbox{strongly in }L^p(\Omega_T)
$$
and also weakly* in $L^\infty(\Omega_T)$, where we set $a_{ij}(u)/u_j^{s_j}:=0$
if $u_j=0$. For those $j$ with $s_j\le 0$, it follows that
$$
  a_{ij}(u^\eps)\to a_{ij}(u), \quad
  \frac{a_{ij}(u^\eps)}{(u_j^\eps)^{s_j}} \to \frac{a_{ij}(u)}{(u_j)^{s_j}}
	\quad\mbox{strongly in }L^p(\Omega_T). 
$$
Furthermore,  $f_i(u^\eps)\to f_i(u)$ strongly in $L^p(\Omega_T)$, for $p<\infty$.
Notice that we use the same notation for $u_i^\eps$ or $(u_i^\eps)^{s_i+1}$  
and the corresponding extensions from $\Omega^\eps$ into $\Omega$  
when considering problem \eqref{2.eq}-\eqref{2.bic}  
defined in the perforated domain $\Omega^\eps$. 

{\em Step 1: problem \eqref{1.eq}-\eqref{1.bic}.}
We use the admissible test function $\phi^\eps=(\phi_1^\eps,\ldots,\phi_n^\eps)$ 
in the weak formulation of \eqref{1.eq}-\eqref{1.bic}, where
$$
  \phi^\eps_i(x,t)= \phi^0_i(x,t)+\eps\phi^1_i(x,t,x/\eps),
	\quad i=1,\ldots,n,
$$
with $\phi^0_i \in C^1([0,T]; H^1(\Omega))$ such that $\phi^0_i(x,T)=0$ and 
$\phi^1_i\in C^1_0(\Omega_T; C^1_{\rm per}(Y))$. This gives
\begin{align}
  0 &= \int_0^T\langle\pa_t u^\eps,\phi^\eps\rangle dt
	+ \int_{\Omega_T}\sum_{i,j=1}^n P^\eps(x)
	\frac{a_{ij}(u^\eps)}{(s_j+1)(u_j^\eps)^{s_j}}\na (u_j^\eps)^{s_j+1}
	\cdot\na\phi_i^\eps dxdt \nonumber \\
	&\phantom{xx}{}- \int_{\Omega_T} f(u^\eps)\cdot\phi^\eps dxdt
	=: I_1^\eps + I_2^\eps + I_3^\eps. \label{ueps}
\end{align}

We perform the limit $\eps\to 0$ in the integrals $I_k^\eps$ term by term, 
for $k=1,2,3$. Using the strong convergence of $(u^\eps)$, we obtain 
\begin{align*}
  \lim_{\eps \to 0} I_1^\eps 
	&= -\lim_{\eps \to 0} \bigg( \int_0^T \int_\Omega u^\eps \cdot 
	(\pa_t \phi^0+ \eps \pa_t \phi^1) dxdt 
	+ \int_\Omega u^\eps(0) \cdot \big(\phi^0(0)+ \eps  \phi^1(0)\big) 
	dx\bigg) \\
  &= -\int_0^T \int_\Omega u \cdot \pa_t \phi^0  dxdt 
	- \int_\Omega u^0 \cdot \phi^0 (0) dx, \\
\lim_{\eps\to 0}I_3^\eps &= -\int_0^T\int_\Omega f(u)\cdot\phi^0 dxdt.
\end{align*}

The limit $\eps\to 0$ in $I_2^\eps$ is more involved.
By \eqref{conv4}, we have $\na (u_j^\eps)^{s_j+1} \rightharpoonup  
\na u_j^{s_j+1} + \nabla_y V_j$ two-scale. 
Furthermore, we deduce from the definition of $P^\eps$, the strong convergence of 
$(u^\eps)$, and the strong two-scale convergences of $(P(x/\eps))$ and 
$(\na\phi_i^\eps)$ that 
$$
  \lim_{\eps\to 0}\bigg\| P^\eps(x)\frac{a_{ij}(u^\eps)}{(u_j^\eps)^{s_j}}
	\na\phi^\eps_i\bigg\|_{L^2(\Omega_T)}
  =|Y|^{-1/2} \bigg\|P(y)\frac{a_{ij}(u)}{u_j^{s_j}}(\na\phi^0_i+\na_y\phi^1_i)
	\bigg\|_{L^2(\Omega_T\times Y)}, 
$$
for $i,j=1, \ldots, n$. Therefore, by Lemma \ref{lem4} in Appendix \ref{app},
$$
  I_2^\eps \to \int_0^T\int_\Omega\dashint_Y \sum_{i,j=1}^n
	P(y)\frac{a_{ij}(u)}{(s_j+1)u_j^{s_j}}
	\big(\na u_j^{s_j+1} +\na_y V_j\big)\cdot
	\big(\na\phi^0_i+\na_y\phi^1_i\big) dydxdt, 
$$
as $\eps \to 0$,  where $\dashint_Y(\cdots)dy=|Y|^{-1}\int_Y(\cdots)dy$. 
Hence, the limit $\eps\to 0$ in \eqref{ueps} leads to
\begin{align}
  -&\int_0^T\int_\Omega u\cdot\pa_t\phi^0 dxdt
	-\int_\Omega u^0\cdot\phi^0(0)dx \nonumber \\
	&\phantom{xx}{}
	+ \int_0^T\int_\Omega\dashint_Y\sum_{i,j=1}^n P(y)\frac{a_{ij}(u)}{(s_j+1)u_j^{s_j}}
	\big(\na u_j^{s_j+1} +\na_y V_j\big)\cdot
	\big(\na\phi^0_i+\na_y\phi^1_i\big) dydxdt  \label{limu}  \\
	&= \int_0^T\int_\Omega f(u)\cdot\phi^0 dxdt. \nonumber
\end{align}

Next, we need to identify $V_j$. For this, let first $\phi^0_i=0$ 
for $i=1,\ldots,n$ in \eqref{limu}. Then
\begin{equation}\label{macro_11}
  0 = \int_0^T\int_\Omega\int_Y \sum_{i,j=1}^n
	P(y)\frac{a_{ij}(u)}{(s_j+1)u_j^{s_j}}
	\big(\na u_j^{s_j+1}+\na_y V_j\big)\cdot\na_y\phi^1_i dydxdt.   
\end{equation}
We  insert the ansatz 
$$
  V_j(t,x,y) = \sum_{k=1}^d\sum_{\ell=1}^n\frac{\pa}{\pa x_k}
	u_\ell^{s_\ell +1}(t,x)W_j^{k\ell}(t,x,y), \quad 
	j = 1, \ldots, n,
$$
with functions $W_j^{k\ell}$, which need to be determined, in \eqref{macro_11}:
\begin{align*}
  0 &= \int_{\Omega_T}\int_Y \sum_{i,j=1}^n
	\frac{a_{ij}(u)}{(s_j+1)u_j^{s_j}} \sum_{m=1}^d P_m(y)\bigg( 
	\frac{\pa u_j^{s_j+1}}{\pa x_m} + \sum_{k=1}^d\sum_{\ell=1}^n
	\frac{\pa u_\ell^{s_\ell+1}}{\pa x_k}\frac{\pa W_j^{k\ell}}{\pa y_m} \bigg) 
	\frac{\pa \phi^1_i}{\pa y_m} dydxdt \\ 
  &= \int_{\Omega_T}\int_Y \sum_{k=1}^d\sum_{\ell=1}^n
	\frac{\pa u_\ell^{s_\ell+1}}{\pa x_k} 
	\sum_{i,j=1}^n\sum_{m=1}^d P_m(y)
	\frac{a_{ij}(u)}{(s_j+1)u_j^{s_j}}\bigg( \frac{\pa W_j^{k\ell}}{\pa y_m}   
	+ \delta_{km}\delta_{j\ell} \bigg)
	\frac{\pa \phi^1_i}{\pa y_m} dydxdt,
\end{align*}
where $\delta_{km}$ is the Kronecker symbol.
By the linear independence of $(\pa u_\ell^{s_\ell+1}/\pa x_k)_{k\ell}$,
we infer that 
$$
  0 = \int_Y  \sum_{i, j=1}^n \sum_{m=1}^d P_m(y)
	\frac{a_{ij}(u)}{(s_j+1)u_j^{s_j}} \bigg(\frac{\pa W_j^{k\ell}}{\pa y_m}    
	+ \delta_{km}\delta_{j\ell}\bigg) 	\frac{\pa \phi^1_i}{\pa y_m} dy, 
$$
for $k=1, \ldots, d$ and $\ell=1, \ldots, n$.  
This means that the functions $W_j^{k\ell}$ are solutions, if they exist, of the
linear elliptic cross-diffusion equations
$$
  \sum_{m=1}^d\sum_{i,j=1}^n\frac{\pa}{\pa y_m}\bigg(P_m(y)\widehat A_{ij}(u(x,t))
	\bigg(\frac{\pa W_j^{k\ell}}{\pa y_m} + \delta_{j\ell}\delta_{km}\bigg)\bigg) = 0,
$$
where
$$
  \widehat A_{ij}(u) = \frac{a_{ij}(u)}{(s_j+1)u_j^{s_j}}, \quad i,j=1,\ldots,n.
$$
More precisely, $W_j^{k\ell}$ are the solutions, if they exist, of the 
elliptic problem
\begin{equation}\label{unit_cell}
\begin{aligned}
  & \diver_y\big(P(y)\widehat A(u(x,t))(\na_y W^{k\ell} + e_ke_\ell)\big) = 0\quad
	\mbox{in }Y, \\
  & \int_Y W^{k\ell}(x,y,t)dy = 0, \quad W_j^{k\ell}\mbox{ is $Y$-periodic},
	\ j=1,\ldots,n,
\end{aligned}
\end{equation}
for $k=1,\ldots,d$ and $\ell=1,\ldots,n$, parametrized by $(x,t)\in \Omega_T$,
where $e_k$ and $e_\ell$ are the standard
basis vectors of $\R^d$ and $\R^n$, respectively, and $e_ke_\ell$ is the matrix
in $\R^{d\times n}$ with the elements $\delta_{km}\delta_{j\ell}$.
The solvability of \eqref{unit_cell} is proved in Lemma \ref{exist_unit_cell} below.

Setting $\phi^1=0$ and arguing similarly as above, 
we can write the macroscopic equations \eqref{limu} as
\begin{align}
  -\int_0^T\int_\Omega & \sum_{i=1}^n  u_i \pa_t \phi_i^0 dxdt 
	+ \int_0^T\int_\Omega\sum_{k,m=1}^d\sum_{i,\ell=1}^n B_{mk}^{i \ell }(u)
	\frac{\pa u_\ell}{\pa x_k}\frac{\pa\phi_i^0}{\pa x_m}dxdt \nonumber \\
	&= \int_0^T\int_\Omega\sum_{i=1}^n f_i(u)\phi_i^0 dxdt 
	+ \int_\Omega \sum_{i=1}^n u_{i}^0 \phi_i^0(0) dx, \label{eq:macro1}
\end{align}
where
\begin{equation}\label{B}
  B_{mk}^{i\ell }(u) = \sum_{j=1}^n\bigg(
	a_{ij}(u)\delta_{km}\delta_{j\ell} \dashint_Y P_m(y)dy
	 + \frac{a_{ij}(u)(s_\ell +1) u_\ell^{s_\ell}}{(s_j+1)u_j^{s_j}}  
	\dashint_Y P_m(y)\frac{\pa W_j^{k \ell}}{\pa y_m}dy\bigg).
\end{equation}

From equation \eqref{eq:macro1} and $u_i^{s_i+1}\in L^2(0, T; H^1(\Omega))$, 
we obtain  $\partial_t u \in L^2(0,T;H^1(\Omega;$ $\R^n)')$. 
This, together with the boundedness of $u_i$, implies that
$\pa_t u_i^{s_i + 1} \in L^2(0,T;H^1(\Omega)')$ for $s_i > 0$. 
Hence, $u_i \in L^2(0,T;H^1(\Omega))\cap H^1(0,T;H^1(\Omega)')$ for those $i$ 
satisfying $-1<s_i \leq 0$ and $u_i^{s_i+1} \in L^2(0,T;H^1(\Omega))\cap 
H^1(0,T;H^1(\Omega)')$ if $s_i>0$.  Therefore, $u_i$, $u_j^{s_j+1} \in 
C^0([0,T];L^2(\Omega))$ for those $i,j=1, \ldots, n$ satisfying 
$-1<s_i\leq 0$ and $s_j >0$. Consequently, the initial datum
is satisfied in the sense of $L^2(\Omega)$.

{\em Step 2: problem \eqref{2.eq}-\eqref{2.bic}.}
We use the two-scale convergence of $\nabla (u^\eps_i)^{s_i+1}$ and take the limit 
$\eps \to 0$ in the weak formulation of \eqref{2.eq}, i.e.   
$$
  \int_0^T\langle\pa_t u^\eps,\phi^\eps\rangle dt
	+ \int_0^T\int_{\Omega^\eps}\sum_{i,j=1}^n
	\frac{a_{ij}(u^\eps)}{(s_j+1)(u_j^\eps)^{s_j}}\na(u_j^\eps)^{s_j+1}
	\cdot\na\phi_i^\eps dxdt  = \int_{\Omega^\eps_T} f(u^\eps)\cdot\phi^\eps dxdt 
$$
to obtain the macroscopic equation
\begin{align}
  &\int_0^T\int_\Omega\dashint_{Y_1} \sum_{i,j=1}^n
	\frac{a_{ij}(u)}{(s_j+1)u_j^{s_j}}\big(\na u_j^{s_j+1} + \na_y V_j\big)\cdot
	\big(\na\phi^0_i+\na_y\phi^1_i\big) dydxdt \nonumber \\
  &\phantom{xx}{}- \int_0^T \int_\Omega  u\cdot \pa_t \phi^0  dxdt  
	- \int_\Omega u^0 \cdot  \phi^0 (0) dx  
	= \int_0^T \int_\Omega  f(u)\cdot \phi^0 dx dt. \label{macro_perfor} 
\end{align}  
Repeating the calculations from Step 1, we arrive at the macroscopic problem 
\eqref{macro_1} with the macroscopic diffusion matrix 
\begin{equation} \label{B_2}
  B_{mk}^{i\ell }(u) = \sum_{j=1}^n\bigg(a_{ij}(u)\delta_{km}\delta_{j\ell} 
	 + \frac{a_{ij}(u)(s_\ell + 1) u_\ell^{s_\ell}}{(s_j+1)u_j^{s_j}}  
	\dashint_{Y_1}\frac{\pa \widehat W_j^{k \ell}}{\pa y_m}dy\bigg). 
\end{equation}
where $\widehat W^{k\ell}$ for $k=1, \ldots, d$ and $\ell=1, \ldots, n$ are the
solutions of the unit-cell problem 
\begin{equation}\label{unit_cell_12}
\begin{aligned}
  & \diver_y\big(\widehat A(u(x,t))(\na_y \widehat W^{k\ell} 
	+ e_{\ell}e_{k})\big) = 0 \quad\mbox{in }Y_1, \quad
	\int_{Y_1}\widehat W^{k\ell} (x,y,t) dy = 0, \\
	& \widehat A(u(x,t))(\na_y \widehat W^{k\ell} + e_{\ell}e_{k})\cdot\nu = 0 
	\quad\text{on }\Gamma, \quad
  \widehat W^{k\ell}_j \mbox{ is }Y\text{-periodic}, 
\end{aligned}
\end{equation}
where $j=1, \ldots, n$. This finishes the proof.
\end{proof}

It remains to prove the solvability of the unit-cell problems.

\begin{lemma}[Solvability of the unit-cell problem] \label{exist_unit_cell}
There exist weak solutions of the unit-cell problems \eqref{unit_cell} and 
\eqref{unit_cell_12}, respectively. The solutions are unique on 
$\{u_i>0:i=1,\ldots,n\}$.
\end{lemma} 

\begin{proof}  
Let us first consider problem \eqref{unit_cell}. Since $\widehat A(u(x,t))$ may vanish,
the unit-cell problem is of degenerate type. Therefore, we introduce the regularization 
\begin{align}
  & \diver_y\big(P(y) \widehat A(u_\delta(x,t))(\na_y W_\delta^{k\ell} 
	+ e_k e_{\ell})\big)	= 0 \quad\text{in } Y, \nonumber \\
	& \int_Y W_{\delta, j}^{k\ell} (x,y,t)dy = 0, \quad
  W^{k\ell}_{\delta,j}\mbox{ is }Y\text{-periodic},\ j =1, \ldots, n, 
	\label{unit_cell_regular}
\end{align}
where $u_{\delta,j}(x,t) = (u_j(x,t)+\delta/2)/(1+\delta)$ for $j=1,\ldots,n$.
Since $0\le u_j(x,t)\le 1$, it follows that 
$0<\delta/(2+2\delta)\le u_{\delta,j}\le (2+\delta)/(2+2\delta)<1$, which
avoids the degeneracy in Assumption A2. Furthermore, we define
$$
  \widetilde W^{k\ell}_{\delta, j}(x,y,t)
	:= \frac{W^{k\ell}_{\delta, j}(x,y,t)}{u_{\delta, j}^{s_j}(x,t)}, 
	\quad j = 1, \ldots, n.  
$$
Then $\widetilde W^{k\ell}_\delta$ satisfies the problem
\begin{align}
  & \diver_y\big(P(y) (A(u_\delta(x,t))\na_y \widetilde W^{k\ell}_\delta 
	+ \widehat A(u_\delta(x,t))e_{k} e_{\ell})\big) = 0 \quad\text{in } Y, \nonumber \\
	& \int_Y \widetilde  W^{k\ell}_{\delta, j}(x,y,t)dy = 0,  \quad
  \widetilde W^{k\ell}_{\delta, j}\mbox{ is }Y\text{-periodic}, \ j = 1,\ldots,n.   
	\label{unit_cell_ref}
\end{align}
Notice that $u_\delta(x,t)$ is independent of $y\in Y$. The weak formulation
of the elliptic problem reads as
\begin{align*}
  0 &= \int_Y \sum_{i,j=1}^n\sum_{m=1}^d P_m(y)\bigg(a_{ij}(u_\delta(x,t))
	\pa_{y_m}\widetilde W_{\delta,j}^{k\ell} 
	+ \widehat a_{ij}(u_\delta(x,t))\delta_{j\ell}\delta_{km}
	 \bigg)\pa_{y_m}\psi_i dy \\
	&= \int_Y \bigg(\sum_{i,j=1}^n P(y)a_{ij}(u_\delta(x,t))
	\na_{y}\widetilde W_{\delta,j}^{k\ell}\cdot\na_y \psi_i
	+ \sum_{i=1}^n P_k(y)\widehat a_{i\ell}(u_\delta(x,t))
	\frac{\pa \psi_i}{\pa y_k}\bigg)dy.
\end{align*}
We take the test function $\psi_i(x,y,t) 
= \sum_{m=1}^n \pa_{im}h(u_\delta(x,t))\phi_m(y)$,
where $\pa_{im} h = \pa^2 h/(\pa \xi_i$ $\pa \xi_m)$ and $\phi_m$ 
is another test function:
$$
  0 = \int_Y \bigg(\sum_{i,j,m=1}^n P(y) \pa_{im}h(u_\delta)a_{ij}(u_\delta)
	\na_y\widetilde W_{\delta,j}^{k\ell}\cdot\na_y\phi_m
	+ \sum_{i,m=1}^n P_k(y) \pa_{im}h(u_\delta)\widehat a_{i\ell}(u_\delta)
	\frac{\pa \phi_m}{\pa y_k}\bigg)dy.
$$
We rename $m\mapsto i$ and $i\mapsto m$ and use the symmetry of the Hessian
$(\pa_{im}h)$:
\begin{align}
  0 &= \int_Y \bigg(\sum_{i,j,m=1}^n P(y) \pa_{im}h(u_\delta)a_{mj}(u_\delta)
	\na_y\widetilde W_{\delta,j}^{k\ell}\cdot\na_y\phi_i
	+ \sum_{i,m=1}^n P_k(y) \pa_{im}h(u_\delta)\widehat a_{m\ell}(u_\delta)
	\frac{\pa \phi_i}{\pa y_k}\bigg)dy \nonumber \\
  &= \int_Y\bigg(\sum_{i,j=1}^n P(y) \big(h''(u_\delta) A(u_\delta)\big)_{ij} 
	\na_y \widetilde W^{k\ell}_{\delta, j}\cdot \na_y \phi_i 
	+ \sum_{i=1}^n P_k(y)\big(h''(u_\delta) \widehat A(u_\delta)\big)_{i\ell}  
	\frac{\pa \phi_i}{\pa y_k}\bigg) dy. \label{weak_regular}
\end{align}
The assumptions on $A(u_\delta)$ and $h(u_\delta)$ imply that 
$h''(u_\delta)A(u_\delta)$ is positive definite in $\Omega_T$, giving
coercivity of the elliptic problem.
Furthermore, for any fixed $\delta>0$, the coefficients of
$h''(u_\delta)A(u_\delta)$ are uniformly bounded. 
Therefore, we can apply the Lax-Milgram lemma to conclude the existence
of a unique solution $\widetilde W^{k\ell}_{\delta}(x,\cdot,t)\in 
H^1_{\rm per}(Y;\R^n)$ of problem \eqref{weak_regular}. As $h''(u_\delta)$ is
invertible, we may consider $\phi=h''(u_\delta)^{-1}\psi$ as a test function
in \eqref{weak_regular}, which means that the function
$W^{k\ell}_{\delta, j}(x,\cdot,t) = u_{\delta,j}^{s_j}(x,t)
\widetilde W^{k\ell}_{\delta, j}(x,\cdot,t)$ for $j=1,\ldots,n$
also solves \eqref{unit_cell_regular}.

The next step is the derivation of bounds uniform in $\delta$.
To this end, we take the test function $\widetilde W^{k\ell}_{\delta}(x,\cdot,t)$
in \eqref{weak_regular}, take into account the lower bound $P_k(y)\ge d_0>0$ for 
$k=1, \ldots, d$, and the definition of $\widehat A(u_\delta(x,t))$, 
and apply the Cauchy-Schwarz inequality. This leads for any $\sigma>0$ to
\begin{align*}
  d_0\int_Y\sum_{j=1}^n u_{\delta,j}^{2s_j}\big|\na_y\widetilde W^{k\ell}_{\delta,j} 
	\big|^2 dy 
  &\leq C_{\sigma}\int_Y\sum_{i,\ell=1}^n
  \frac{(h''(u_\delta)A(u_\delta))_{i\ell}^2}{u_{\delta, i}^{2s_i} 
	u_{\delta,\ell}^{2s_\ell}}dy  
	+ \sigma\int_Y\sum_{j=1}^n u_{\delta, j}^{2s_j} 
	|\na_y \widetilde W^{k\ell}_{\delta, j} |^2 dy.
\end{align*}
Choosing $\sigma=d_0/2$ and using Assumption A6, we find that
$$
  \frac{d_0}{2}\int_Y\sum_{j=1}^n\big|\na_y W^{k\ell}_{\delta,j} \big|^2 dy
  \leq C_{d_0/2}\int_Y\sum_{i,\ell=1}^n
  \frac{(h''(u_\delta)A(u_\delta))_{i\ell}^2}{u_{\delta, i}^{2s_i} 
	u_{\delta,\ell}^{2s_\ell}}dy \le C,
$$
where $C>0$ does not depend on $\delta$. As the mean of $W_{\delta,j}^{k\ell}$
vanishes, the Poincar\'e-Wirtinger inequality gives a uniform estimate in
$H^1(Y)$.

The uniform estimate for $W^{k\ell}_{\delta}$
implies the existence of a subsequence, which is not relabeled, such that
$W^{k\ell}_\delta \rightharpoonup W^{k\ell}$ weakly in $H^1(Y;\R^n)$ as $\delta\to 0$,
for $x\in \Omega$ and $t>0$. Hence, we can pass to the limit $\delta\to 0$ in
\eqref{unit_cell_regular} to conclude that $W^{k\ell}$ is a solution of  
\eqref{unit_cell}. 

We claim that the solution is unique on the set $\{(x,t):u_i(x,t)>0$ for 
$i=1,\ldots,n\}$. Indeed, taking two
solutions $W^{k\ell}_{(1)}$ and $W^{k\ell}_{(2)}$ of \eqref{unit_cell}, choosing
$(x,t)$ such that $u_i(x,t)>0$ and arguing as before, we obtain
$$
  \|\na_y (W^{k\ell}_{(1),i} - W^{k\ell}_{(2),i})\|_{L^2(Y;\R^{d})} \leq 0
$$
for $k=1,\ldots,d$ and $i,\ell=1,\ldots,n$. This implies that
$W^{k\ell}_{(1)}=W^{k\ell}_{(2)}$ and proves the claim.  

The same arguments ensure also the existence of a solution of the unit-cell problem
\eqref{unit_cell_12} and its uniqueness for those $x\in\Omega$ and $t>0$ 
satisfying $u_i(t,x)>0$ for all $i=1,\ldots,n$. 
\end{proof}


\section{Proof of Theorem \ref{main_filling}}\label{sec_proof2}
First, we state an existence result which follows from \cite{ZaJu17}.

\begin{lemma}[Entropy inequality] \label{Lemma_entropy}
There exists a weak solution $u^\eps=(u_1^\eps,\ldots,u_n^\eps)$ of 
problem \eqref{1.eq}-\eqref{1.bic} with the diffusion matrix \eqref{1.nonloc} 
in the sense of Definition \ref{def.weak}. This solution satisfies the
entropy inequality
\begin{equation}\label{4.epi}
  \int_\Omega h(u^\eps)dx + C\int_0^T\int_\Omega\sum_{i=1}^n \big(u_{n+1}^\eps
	|\na(u_i^\eps)^{1/2}|^2 + |\na(u_{n+1}^\eps)^{1/2}|^2\big)dxdt
	\le \int_\Omega h(u^0)dx,
\end{equation}
where $C=d_0\min_{i=1,\ldots,n}D_i$.
A similar estimate with $\Omega$ replaced by $\Omega^\eps$ holds for solutions of 
problem \eqref{2.eq}-\eqref{2.bic} with the diffusion matrix \eqref{1.nonloc}.
\end{lemma}

\begin{proof}[Proof ideas]
The existence of a weak solution $u^\eps$ follows from Theorem 1 in \cite{ZaJu17}
for $p_i(u)=D_i$ ($i=1,\ldots,n$) and $q(s)=s$.  The entropy inequality 
\eqref{4.epi} follows from inequality (33) in \cite{ZaJu17} 
in the regularization limit.  A direct proof of estimate \eqref{4.epi} using 
the definition of a weak solution of \eqref{1.eq}-\eqref{1.bic} 
or \eqref{2.eq}-\eqref{2.bic} with the 
diffusion matrix \eqref{1.nonloc} can be found in Appendix~\ref{sec.lemma}.  
\end{proof}

\begin{lemma}[A priori estimates]\label{lem.est2}
Weak solutions of \eqref{1.eq}-\eqref{1.bic} with 
diffusion matrix \eqref{1.nonloc} satisfy 
\begin{align}
  & \|(u_{n+1}^\eps)^{1/2} u_i^\eps\|_{L^2(0,T;H^1(\Omega))} \le C, \nonumber \\
	& \|(u_{n+1}^\eps)^{1/2}\|_{L^2(0,T;H^1(\Omega))}
	+ \|(u_{n+1}^\eps)^{3/2}\|_{L^2(0,T;H^1(\Omega))} \le C, \nonumber \\
	& \|\vartheta_\tau u_{n+1}^\eps - u_{n+1}^\eps\|_{L^{5/2}(\Omega_T)}
	\le C\tau^{1/5}, \label{apriori_2} \\
	& \|\vartheta_\tau((u_{n+1}^\eps)^{1/2}u_i^\eps)
	- (u_{n+1}^\eps)^{1/2}u_i^\eps\|_{L^2(\Omega_{T-\tau})}
	\le C\tau^{1/10}, \nonumber
\end{align}
for  all $\eps>0$ and $i=1,\ldots,n$, where  $\vartheta_\tau v(x,t)=v(x,t+\tau)$
for $t \in (0, T-\tau)$ and $\tau \in (0, T)$ and  
the constant $C>0$ is independent of $\eps$. 
\end{lemma}

\begin{proof}
The entropy production inequality \eqref{4.epi}
shows that there exists $C>0$ independent of $\eps$ such that for all $i=1,\ldots,n$,
\begin{equation}\label{estim_vf_11}
  \|(u_{n+1}^\eps)^{1/2}\|_{L^2(0,T;H^1(\Omega))}
	+ \|(u_{n+1}^\eps)^{1/2}\na (u_i^\eps)^{1/2}\|_{L^2(\Omega_T)} \le C.
\end{equation}
Because of 
$$
  \na\big((u_{n+1}^\eps)^{1/2}u_i^\eps\big)
	= 2(u_i^\eps u_{n+1}^\eps)^{1/2}\na (u_i^\eps)^{1/2}
	+ u_i^\eps\na(u_{n+1}^\eps)^{1/2},
$$
estimate \eqref{estim_vf_11}, and the  boundedness of $u_i^\eps$  
for $i=1, \ldots, n$, we conclude that
$$
  \|\na((u_{n+1}^\eps)^{1/2}u_i^\eps)\|_{L^2(\Omega_T)} \le C,
	\quad i=1,\ldots,n.
$$
Adding this inequality for $i=1,\ldots,n$ and recalling that
$\sum_{i=1}^n \na u_i^\eps = -\na u_{n+1}^\eps$, it follows that
$$
  \|(u_{n+1}^\eps)^{3/2}\|_{L^2(0,T;H^1(\Omega))} \le C.
$$

It remains to verify the uniform estimates on the equicontinuity of $u^\eps$ 
with respect to the time variable.
For this, we define similarly as in the proof of Lemma~\ref{lem.est}
$$
  \phi(x,t) = \int_{t-\tau}^t\big((\vartheta_\tau u_{n+1}^\eps)^{3/2}
	- (u_{n+1}^\eps)^{3/2}\big)\kappa(\sigma)d\sigma
$$
for some $\tau\in(0,T)$, where $\kappa(\sigma)=1$ for $\sigma\in(0,T-t)$ and
$\kappa(\sigma)=0$ for $\sigma\in[-\tau,0]\cup[T-\tau,T]$. We take $\phi$
as a test function in the sum of equations \eqref{weak_2} for $i=1,\ldots,n$
and use Lemma~\ref{lem.ineq} with $p=3/2$ and the Cauchy-Schwarz inequality
to infer that
\begin{align*}
  &\|\vartheta_\tau u_{n+1}^\eps - u_{n+1}^\eps\|_{L^{5/2}(\Omega_{T-\tau})}^{5/2}
  \le \int_0^{T-\tau}\int_\Omega|\vartheta_\tau u_{n+1}^\eps - u_{n+1}^\eps|^{5/2}
	dxdt \\
	&\le \int_0^{T-\tau}\int_\Omega(\vartheta_\tau u_{n+1}^\eps - u_{n+1}^\eps)
	\big(\vartheta_\tau (u_{n+1}^\eps)^{3/2} - (u_{n+1}^\eps)^{3/2}\big)dxdt \\
	&\le C\bigg|\int_0^{T-\tau}\int_\Omega\bigg(\int_t^{t+\tau}P^\eps(x)\sum_{i=1}^n
	(A(u^\eps) \na u^\eps)_i\,  d\sigma\bigg)\cdot\na\big(\vartheta_\tau (u_{n+1}^\eps)^{3/2}
	- (u_{n+1}^\eps)^{3/2}\big)dxdt\bigg| \\
	&\le C\bigg\{\int_{\Omega_{T-\tau}} \hspace{-0.1 cm}\bigg[\int_t^{t+\tau} \sum_{i=1}^n
	(A(u^\eps)\na u^\eps)_i d\sigma\bigg]^2 dxdt\bigg\}^{\frac 1 2} 
	\bigg\{\int_{\Omega_{T-\tau}} \hspace{-0.3 cm} \big|\na\big(\vartheta_\tau (u_{n+1}^\eps)^{\frac 3 2}
	- (u_{n+1}^\eps)^{\frac 3 2}\big)\big|^2dxdt\bigg\}^{\frac 1 2}.
\end{align*}
The second factor on the right-hand side is uniformly bounded since
$\na(u_{n+1}^\eps)^{3/2}$ is bounded in $L^2(\Omega_T)$.
The first factor can be estimated from above by using
definition \eqref{1.nonloc} of $A(u^\eps)$ and the uniform estimates for 
$(u_{n+1}^\eps)^{1/2}\na u_i^\eps$ as well as $\na(u_{n+1}^\eps)^{1/2}$:
\begin{align*}
  & \int_0^{T-\tau}\int_\Omega\bigg(\int_{t}^{t+\tau} \sum_{i=1}^n
	(A(u^\eps)\na u^\eps)_i d\sigma\bigg)^2 dxdt \\
	&\le C\tau\int_0^{T-\tau}\int_\Omega u_{n+1}^\eps\sum_{i=1}^n
	\Big(\big|\na\big((u_{n+1}^\eps)^{1/2} u_i^\eps\big)\big|^2
	+ |\na(u_{n+1}^\eps)^{1/2}|^2 (u_i^\eps)^2\Big)dxdt \le C\tau.
\end{align*}
We conclude that
$$
  \|\vartheta_\tau u_{n+1}^\eps - u_{n+1}^\eps\|_{L^{5/2}(\Omega_{T-\tau})} 
	\le C\tau^{1/5}.
$$

To prove the remaining estimate in \eqref{apriori_2}, we take the test function
$$
  \phi_i(x,t) = \int_{t-\tau}^t\vartheta_\tau((u_{n+1}^\eps)^{1/2})
	\big(\vartheta_\tau((u_{n+1}^\eps)^{1/2}u_i^\eps)
	- (u_{n+1}^\eps)^{1/2}u_i^\eps\big)\kappa(\sigma)d\sigma
$$
in \eqref{weak_2} for $i=1,\ldots,n$. A computation shows that
\begin{align*}
  & \int_0^{T-\tau}\int_\Omega\sum_{i=1}^n
	\big(\vartheta_\tau((u_{n+1}^\eps)^{1/2}u_i^\eps)
	- (u_{n+1}^\eps)^{1/2}u_i^\eps\big)^2 dxdt \\
	&\phantom{xx}{}+ \int_0^{T-\tau}\int_\Omega\int_t^{t+\tau}\sum_{i=1}^n 
	P^\eps(x) D_i 
	\Big[(u_{n+1}^\eps)^{1/2}\na\big((u_{n+1}^\eps)^{1/2}u_i^\eps\big)
	- 3u_i^\eps(u_{n+1}^\eps)^{1/2}\na(u_{n+1}^\eps)^{1/2}\Big] \\
	&\phantom{xx}{}\times\na\Big[\vartheta_\tau((u_{n+1}^\eps)^{1/2})
	\big(\vartheta_\tau((u_{n+1}^\eps)^{1/2}u_i^\eps)
	- (u_{n+1}^\eps)^{1/2}u_i^\eps\big)\Big]d\sigma dxdt \\
	&= \int_0^{T-\tau}\int_\Omega\sum_{i=1}^n u_i^\eps
	\big(\vartheta_\tau (u_{n+1}^\eps)^{1/2} - (u_{n+1}^\eps)^{1/2}\big)
	\big(\vartheta_\tau((u_{n+1}^\eps)^{1/2}u_i^\eps)
	- (u_{n+1}^\eps)^{1/2}u_i^\eps\big)dxdt \\
	&\le \sum_{i=1}^n\big\|\vartheta_\tau (u_{n+1}^\eps)^{1/2} - (u_{n+1}^\eps)^{1/2}
	\big\|_{L^2(\Omega_{T-\tau})}\big\|\vartheta_\tau((u_{n+1}^\eps)^{1/2}u_i^\eps)
	- (u_{n+1}^\eps)^{1/2}u_i^\eps\big\|_{L^2(\Omega_{T-\tau})}.
\end{align*}
The second integral on the left-hand side is bounded by $C\tau^{1/2}$ in view of
the gradient estimates in \eqref{apriori_2}.
We infer from Lemma~\ref{lem.ineq} with $p=2$, $a=\vartheta_\tau (u_{n+1}^\eps)^{1/2}$,
$b=(u_{n+1}^\eps)^{1/2}$ and the third estimate in \eqref{apriori_2} that
$$
  \|\vartheta_\tau (u_{n+1}^\eps)^{1/2} - (u_{n+1}^\eps)^{1/2}
	\|_{L^2(\Omega_{T-\tau})} 
	\le C\|\vartheta_\tau u_{n+1}^\eps - u_{n+1}^\eps\|_{L^2(\Omega_{T-\tau})}^{1/2}
	\le C\tau^{1/10},
$$
finishing the proof.
\end{proof}

\begin{remark}\rm
Similar uniform estimates as in Lemma \ref{lem.est2} hold for the solutions
of problem \eqref{2.eq}-\eqref{2.bic} with the
diffusion matrix \eqref{1.nonloc} defined in a perforated domain with
the only difference that the domain $\Omega$ has to be replaced by $\Omega^\eps$:
\begin{align}
  & \big\|(u_{n+1}^\eps)^{1/2} u_i^\eps\big\|_{L^2(0,T;H^1(\Omega^\eps))} 
	\le C, \nonumber \\
	& \|(u_{n+1}^\eps)^{1/2}\|_{L^2(0,T;H^1(\Omega^\eps))}
	+ \|(u_{n+1}^\eps)^{3/2}\|_{L^2(0,T;H^1(\Omega^\eps))} \le C, \nonumber \\
	& \big\|\vartheta_\tau u_{n+1}^\eps - u_{n+1}^\eps\big\|_{L^{5/2}(\Omega_T^\eps)}
	\le C\tau^{1/5}, \label{apriori_22} \\
	& \big\|\vartheta_\tau\big((u_{n+1}^\eps)^{1/2}u_i^\eps\big)
	- (u_{n+1}^\eps)^{1/2}u_i^\eps\big\|_{L^2(\Omega^\eps_{T-\tau})}
	\le C\tau^{1/10}, \nonumber
\end{align}
for $i=1,\ldots,n$ and $\Omega_T^\eps=\Omega^\eps\times(0,T)$. 
\qed
\end{remark}

The uniform estimates in Lemma~\ref{lem.est2} yield the following convergence
results.

\begin{lemma}[Convergence]\label{lem.conv2}
Let $u^\eps$ be a solution of \eqref{1.eq}-\eqref{1.bic} 
with diffusion matrix \eqref{1.nonloc} satisfying
estimates \eqref{apriori_2}. Then there exist functions $u_1,\ldots,u_n 
\in L^\infty(0,T; L^\infty(\Omega))$, 
with  $u_{n+1}^{1/2}u_i$, $u_{n+1}^{1/2}\in L^2(0,T;$ $H^1(\Omega))$, and functions
$V_1,\ldots,V_{n+1}\in L^2(\Omega_T;H^1_{\rm per}(Y)/\R)$ such that, up
to subsequences,
\begin{align}
  u^\eps_{n+1} &\to u_{n+1} && \text{strongly in }L^p(\Omega_T), \;  
	p \in (1, \infty), \nonumber \\ 
  u^\eps_i &\rightharpoonup u_i && \text{weakly in }L^p(\Omega_T),\ p\in(1,\infty), 
	\nonumber \\ 
  (u_{n+1}^\eps)^{1/2} &\rightharpoonup(u_{n+1})^{1/2} &&\text{weakly  in } 
	L^2(0,T; H^1(\Omega)), \nonumber \\
  (u_{n+1}^\eps)^{1/2} u^\eps_i &\to (u_{n+1})^{1/2} u_i && 
	\text{strongly in } L^2(\Omega_T), \label{4.conv2} \\
  (u_{n+1}^\eps)^{1/2} u^\eps_i &\rightharpoonup (u_{n+1})^{1/2} u_i && 
	\text{weakly  in } L^2(0,T; H^1(\Omega)), \nonumber \\
  \na ((u_{n+1}^\eps)^{1/2} u^\eps_i) &\rightharpoonup\nabla((u_{n+1})^{1/2}u_i)  
	+ \na_y V_i  && \text{two-scale}, \nonumber \\
  \na(u_{n+1}^\eps)^{1/2} &\rightharpoonup \na (u_{n+1})^{1/2} + \na_y V_{n+1}  
	&& \text{two-scale}, \nonumber 
\end{align}
as $\eps \to 0$, where $i=1,\ldots, n$ and $u_{n+1} = 1 - \sum_{i=1}^n u_i$. 
\end{lemma}

\begin{proof}
The estimates for $u_{n+1}^\eps$ in \eqref{apriori_2} and Lemma~\ref{lem.ineq}
with $p=2$ show that
$$
  \|\vartheta_\tau(u_{n+1}^\eps)^{1/2} - (u_{n+1}^\eps)^{1/2}\|_{L^2(\Omega_{T-\tau})}
	\le \|\vartheta_\tau u_{n+1}^\eps - u_{n+1}^\eps\|_{L^2(\Omega_{T-\tau})}^{1/2}
	\le C\tau^{1/10}.
$$
Thus, together with the uniform bound for $u_{n+1}^\eps$ in $L^2(0,T;H^1(\Omega))$,
the Aubin-Lions lemma \cite{Sim87} implies the existence of a function $w\in L^2(\Omega_T)$ and 
a subsequence (not relabeled) such that $(u_{n+1}^\eps)^{1/2}\to w$ strongly in 
$L^2(\Omega_T)$ as $\eps\to 0$. In particular,
possibly for another subsequence, $(u_{n+1}^\eps)^{1/2}\to w$ a.e.\ in 
$\Omega_T$. Then, defining $u_{n+1}:=w^2\ge 0$, it follows that
$u_{n+1}^\eps\to u_{n+1}$ a.e.\ in $\Omega_T$ and, because of the boundedness
of $u_{n+1}^\eps$, also $u_{n+1}^\eps\to u_{n+1}$ in $L^p(\Omega_T)$ for any $p<\infty$.

The weak convergence of $(u_i^\eps)$ to $u_i$ in $L^p(\Omega_T)$ 
for $p<\infty$ is a consequence
of the uniform $L^\infty$-bound of $u_i^\eps$. As a consequence,
$(u_{n+1}^\eps)^{1/2}u_i^\eps\rightharpoonup u_{n+1}^{1/2}u_i$ weakly in
$L^2(\Omega_T)$. By the first estimate
in \eqref{apriori_2}, a subsequence of $((u_{n+1}^\eps)^{1/2}u_i^\eps)$
is weakly converging in $L^2(0,T;H^1(\Omega))$, and we can identify the
limit by $u_{n+1}^{1/2}u_i$. In fact, this limit is strong because
the first and last estimate in \eqref{apriori_2} allow us to apply the Aubin-Lions 
lemma again to conclude that, for a subsequence,
$(u_{n+1}^\eps)^{1/2}u_i^\eps\to u_{n+1}^{1/2}u_i$ strongly in $L^2(\Omega_T)$.

Using the first three estimates in \eqref{apriori_2} and the compactness theorem
for two-scale convergence (see Lemma \ref{lem1} in Appendix \ref{app}),
we obtain the two-scale convergences in \eqref{4.conv2}.
\end{proof}

The uniform estimates \eqref{apriori_22} lead to the following convergences
for the extensions $\overline{u_i^\eps}$, $\overline{(u_{n+1}^\eps)^{1/2}}$,
and $\overline{(u_{n+1}^\eps)^{1/2}u_i^\eps}$ from $\Omega^\eps$ to $\Omega$
of  $u_i^\eps$, $(u_{n+1}^\eps)^{1/2}$, and $(u_{n+1}^\eps)^{1/2}u_i^\eps$, 
respectively, where $i=1,\ldots,n$ and  $u^\eps$  is a weak solution of problem  
\eqref{2.eq} and \eqref{2.bic} with the diffusion matrix~\eqref{1.nonloc}. 

For any $\psi\in L^p(\Omega_T^\eps)$,
we denote by $[\psi]^\sim$ the extension of $\psi$ by zero from $\Omega^\eps_T$
to $\Omega_T$.

\begin{lemma}[Convergence]\label{lem.conv_2}
Let $u^\eps$ be a solution of \eqref{2.eq} and \eqref{2.bic} with 
the diffusion matrix~\eqref{1.nonloc}, satisfying
estimates \eqref{apriori_22}. Then there exist  $u_1,\ldots,u_n 
\in L^\infty(0,T; L^\infty(\Omega))$
with $u_{n+1}^{1/2}u_i$, $u_{n+1}^{1/2} \in L^2(0,T;H^1(\Omega))$ and functions
$V_1,\ldots,V_{n+1}\in L^2(\Omega_T;H^1_{\rm per}(Y_1)/\R)$ such that, up
to subsequences,
\begin{align}
  \overline{u^\eps_{n+1}} &\to  u_{n+1} && \text{strongly  in } L^2(\Omega_T), 
	\nonumber \\ 
  \overline{(u_{n+1}^\eps)^{1/2}} &\rightharpoonup  u_{n+1}^{1/2}   
	&& \text{weakly in } L^2(0,T; H^1(\Omega)), \nonumber  \\
  [u^\eps_i]^{\sim}&\rightharpoonup\theta u_i &&\text{weakly in }L^p(\Omega_T),\ 
	p \in (1, \infty), \nonumber \\ 
  \overline{(u_{n+1}^\eps)^{1/2} u^\eps_i}&\to u_{n+1}^{1/2} u_i
	&& \text{strongly in } L^2(\Omega_T), \nonumber \\
  \overline{(u_{n+1}^\eps)^{1/2} u^\eps_i}&\rightharpoonup 
	u_{n+1}^{1/2} u_i && \text{weakly in } L^2(0,T; H^1(\Omega)), \label{conv_2} \\
  [(u_{n+1}^\eps)^{1/2} u^\eps_i]^{\sim} 
	&\rightharpoonup \chi_{Y_1}u_{n+1}^{1/2} u_i && \text{two-scale}, \nonumber \\
  [\na ((u_{n+1}^\eps)^{1/2} u^\eps_i)]^{\sim} &\rightharpoonup \chi_{Y_1}
	(\na(u_{n+1}^{1/2} u_i)  + \na_y V_i) && \text{two-scale}, \nonumber \\
  [\na(u_{n+1}^\eps)^{1/2}]^{\sim} &\rightharpoonup \chi_{Y_1}(\na u_{n+1}^{1/2}   
	+ \na_y V_{n+1})   && \text{two-scale},  \nonumber 
\end{align}
for $i=1,\ldots,n$, $u_{n+1}=1-\sum_{i=1}^n u_i$, $\theta=|Y_1|/|Y|$, and
$\chi_{Y_1}$ is the characteristic function of $Y_1$.
\end{lemma}

\begin{proof}
As in the proof of Lemma \ref{lem.est2}, we obtain the uniform estimate
$$
  \|\vartheta_\tau (u_{n+1}^\eps)^{1/2} - (u_{n+1}^\eps)^{1/2}
	\|_{L^2(\Omega^\eps_{T-\tau})} \le C\tau^{1/10}.
$$
Then, together with the uniform bound on $(u_{n+1}^\eps)^{1/2}$ in
$L^2(0,T;H^1(\Omega^\eps))$,  the properties of
the extension of $(u_{n+1}^\eps)^{1/2}$ from $\Omega^\eps$ to $\Omega$, and the Aubin-Lions lemma \cite{Sim87},  we conclude the strong convergence (up to a subsequence)
$$
  \overline{(u_{n+1}^\eps)^{1/2}} \to \overline{w}\quad\text{strongly in }
	L^2(\Omega_T)
$$
as $\eps\to 0$. To identify the limit, we use the properties of the extension,
the boundedness of $u_{n+1}^\eps$, and the elementary inequality
$|a-b|\le 2|\sqrt{a}-\sqrt{b}|$ for $0\leq a,b \leq1$ 
(also see Lemma~\ref{lem.ineq}) to find that
\begin{align*}
  \big\|\overline{u_{n+1}^{\eps_m}}-\overline{u_{n+1}^{\eps_k}}\big\|_{L^2(\Omega_T)}
	&\le C\|u_{n+1}^{\eps_m}-u_{n+1}^{\eps_k}\|_{L^2(\Omega_T^\eps)}
	\le C\big\|(u_{n+1}^{\eps_m})^{1/2}-(u_{n+1}^{\eps_k})^{1/2}
	\big\|_{L^2(\Omega_T^\eps)} \\
	&\le C\big\|\overline{(u_{n+1}^{\eps_m})^{1/2}}
	- \overline{(u_{n+1}^{\eps_k})^{1/2}}\big\|_{L^2(\Omega_T)},
\end{align*}
for a sequence $(\eps_n)_{n \in \mathbb N}$. Thus, the strong convergence of
$\overline{(u_{n+1}^\eps)^{1/2}}$ in $L^2(\Omega_T)$ implies the strong convergence
$\overline{u_{n+1}^\eps}\to u_{n+1}$ in $L^2(\Omega_T)$. Then the weak convergence
\begin{align*}
  \theta\int_{\Omega_T} u_{n+1}^{1/2}\phi dxdt
	&= \lim_{\eps\to 0}\int_{\Omega_T}(\overline{u_{n+1}^\eps})^{1/2}\chi_{\Omega^\eps}
	\phi dxdt 
	= \lim_{\eps\to 0}\int_{\Omega_T}(u_{n+1}^\eps)^{1/2}\chi_{\Omega^\eps}\phi dxdt \\
	&= \lim_{\eps\to 0}\int_{\Omega_T}\overline{(u_{n+1}^\eps)^{1/2}}\chi_{\Omega^\eps}
	\phi dxdt = \theta\int_{\Omega_T} \overline{w}\phi dxdt
\end{align*}
for any $\phi\in C_0(\Omega_T)$ shows that $\overline{w}=u_{n+1}^{1/2}$ a.e.\
in $\Omega_T$. We have proved the first two convergences in \eqref{conv_2}.

The uniform estimate for $(\na(u_{n+1}^\eps)^{1/2})$ and the
compactness results for the two-scale convergence, see, e.g., \cite{All92} or  
Lemma~\ref{lem12} in Appendix \ref{app}, imply the last convergence in \eqref{conv_2}.
Moreover, by the first and last estimate in \eqref{apriori_22} for
$(u_{n+1}^\eps)^{1/2}u_i^\eps$, the properties of its extension from
$\Omega^\eps$ to $\Omega$, and the Aubin-Lions lemma, it follows that,
up to a subsequence,
$\overline{(u_{n+1}^\eps)^{1/2}u_i^\eps}\to v_i$ strongly in
$L^2(\Omega_T)$ and weakly in $L^2(0,T;H^1(\Omega))$. We need to identify
this limit.
To this end, we first observe that, thanks to the boundedness 
of $u_i^\eps$ in $\Omega_T^\eps$, it follows that
$$
  |u_i^\eps]^\sim \rightharpoonup \chi_{Y_1} u_i\quad\mbox{two-scale}
$$
for some function $u_i\in L^p(\Omega_T\times Y)$, 
where $p \in (1, \infty)$ and $i=1, \ldots, n$.  The a priori estimates
and the compactness properties for sequences defined in perforated domains, 
see \cite{All92} or Lemma~\ref{lem12} in Appendix \ref{app},  
yield  the existence of functions $V_1,\ldots,V_n\in
L^2(\Omega_T;H^1_{\rm per}(Y_1)/\R)$ such that, up to subsequences,
\begin{align*}
  [(u_{n+1}^\eps)^{1/2}u_i^\eps]^\sim &\rightharpoonup \chi_{Y_1} v_i
	&&\text{two-scale}, \\	
	[\na((u_{n+1}^\eps)^{1/2}u_i^\eps)]^\sim &\rightharpoonup
	\chi_{Y_1}(\na v_i + \na_y V_i) &&\text{two-scale}. 
\end{align*}
The strong convergence of $\overline{(u_{n+1}^\eps)^{1/2}u_i^\eps}$ and 
the identity 
$$
  \int_{\Omega_T} [(u_{n+1}^\eps)^{1/2}u_i^\eps]^\sim \phi dx dt 
  = \int_{\Omega_T} (u_{n+1}^\eps)^{1/2}u_i^\eps\chi_{\Omega^\eps} \phi dx dt 
  = \int_{\Omega_T} \overline{(u_{n+1}^\eps)^{1/2}u_i^\eps}\chi_{\Omega^\eps} \phi 
	dx dt 
$$
for any $\phi \in C_0(\Omega_T)$ imply that
$$	
  [(u_{n+1}^\eps)^{1/2}u_i^\eps]^\sim \rightharpoonup \theta v_i
	\quad \text{weakly in } L^2(\Omega_T).
$$
By Proposition \ref{prop.unfold} and Theorem \ref{conv_unfold} in Appendix
\ref{app}, this gives
\begin{align}
  \T_{Y_1}^\eps\big((u_{n+1}^\eps)^{1/2}u_i^\eps\big) &\rightharpoonup v_i
	&&\text{weakly in }L^2(\Omega_T\times Y_1), \label{T1} \\
  \T_{Y_1}^\eps\big(\na((u_{n+1}^\eps)^{1/2}u_i^\eps)\big) &\rightharpoonup
	\na v_i + \na_y V_i &&\text{weakly in }L^2(\Omega_T\times Y_1). \nonumber
\end{align}
The strong convergence of $\overline{(u_{n+1}^\eps)^{1/2}}$,  the two-scale
convergence of $[u_i^\eps]^\sim$, and the fact that 
$\overline{(u_{n+1}^\eps)^{1/2}}\chi_{\Omega^\eps} 
= (u_{n+1}^\eps)^{1/2} \chi_{\Omega^\eps}$, imply 
$$
  \T_{Y_1}^\eps\big((u_{n+1}^\eps)^{1/2}u_i^\eps\big)
	= \T_{Y_1}^\eps\big((u_{n+1}^\eps)^{1/2}\big)\T_{Y_1}^\eps(u_i^\eps)
	\rightharpoonup u_{n+1}^{1/2}u_i\quad\mbox{weakly in }L^2(\Omega_T\times Y_1).
$$
By the convergence \eqref{T1} and the fact that $u_{n+1}$ and $v_i$ are  
independent of $y$, we infer that $u_i(x,y,t)=u_i(x,t)$ and $v_i=u_{n+1}^{1/2}u_i$, 
proving the claim.
\end{proof}

\begin{proof}[Proof of Theorem \ref{main_filling}]
Let $\phi^0\in C_0^1([0,T];C^1(\overline{\Omega};\R^n))$ and
$\phi^1\in C_0^1(\Omega_T;C^1_{\rm per}(Y;\R^n))$ and set
$\phi(x,t) = \phi^0(x,t) + \eps\phi^1(x,x/\eps,t)$. 
We take this function as a test function in \eqref{weak_2} and pass to the
limit $\eps\to 0$, using the two-scale convergence of 
$\na((u_{n+1}^\eps)^{1/2}u_i^\eps)$ and $\na (u_{n+1}^\eps)^{1/2}$
(the last two convergences in \eqref{4.conv2}):
\begin{align}
  0 &= -\int_0^T\int_\Omega u\cdot\pa_t \phi^0 dxdt
	+ \int_0^T\int_\Omega\dashint_{Y}\sum_{i=1}^n P(y)D_i \, u_{n+1}^{1/2} \nonumber \\
	&\phantom{xx}{}\times\Big(\na(u_{n+1}^{1/2}u_i) + \na_y V_i 
	- 3u_i\big(\na u_{n+1}^{1/2}
	+ \na_y V_{n+1}\big)\Big)\cdot(\na \phi_i^0 + \na_y\phi_i^1)\, dydxdt. \label{wf1}
\end{align}
Choosing $\phi^0=0$ and setting $W_i = V_i - 3u_iV_{n+1}$, this gives
$$
  0 = \sum_{i=1}^n\int_0^T\int_\Omega\dashint_{Y} P(y)D_i u_{n+1}^{1/2} \\
	\big(\na(u_{n+1}^{1/2}u_i) - 3u_i\na u_{n+1}^{1/2}
	+ \na_y W_i\big)\cdot\na_y\phi_i^1 dydxdt.
$$
This is a linear equation
for $W_1,\ldots,W_n$ and a weak formulation of a system  of uncoupled elliptic 
equations for $W= (W_1,\ldots,W_n)$.  Since for  $x\in \Omega$ and $t>0$ such that 
$u_{n+1}(t,x) > 0$, we have a unique (up to a constant) solution of the system for  
$W$, each $W_i$ is defined by 
\begin{equation}\label{eq.Wi}
  0 = \int_0^T\int_\Omega\dashint_{Y} P(y)D_i u_{n+1}^{1/2} \\
	\big(\na(u_{n+1}^{1/2}u_i) - 3u_i\na u_{n+1}^{1/2}
	+ \na_y W_i\big)\cdot\na_y\phi_i^1 dydxdt.
\end{equation}
This motivates the following ansatz:
\begin{equation}\label{Wi}
  W_i(x,y,t) = \sum_{\ell=1}^d\bigg(\frac{\pa}{\pa x_\ell}(u_{n+1}^{1/2}u_i) 
	- 3u_i\frac{\pa}{\pa x_\ell}u_{n+1}^{1/2}\bigg)w_i^\ell(x,y,t)
\end{equation}
for some functions $w_i^\ell$ for $\ell=1,\ldots,d$ and $i=1, \ldots, n$.
Substituting the ansatz \eqref{Wi} into \eqref{eq.Wi}, we find that $w_i^\ell$ solves
\begin{align*}
  0 &= D_i\int_0^T\int_\Omega\dashint_Y u_{n+1}^{1/2}\sum_{k=1}^d P_k(y)
	\bigg\{\frac{\pa}{\pa x_k}(u_{n+1}^{1/2}u_i) - 3u_i\frac{\pa}{\pa x_k}u_{n+1}^{1/2} \\
	&\phantom{xx}{}+ \sum_{\ell=1}^d \bigg(\frac{\pa}{\pa x_\ell}(u_{n+1}^{1/2}u_i) 
	- 3u_i\frac{\pa}{\pa x_\ell}u_{n+1}^{1/2}\bigg)\frac{\pa w_i^\ell}{\pa y_k}
	\bigg\}\frac{\pa\phi_i^1}{\pa y_k} dydxdt \\
	&= D_i\sum_{\ell=1}^d \int_{\Omega_T}\dashint_Y u_{n+1}^{1/2}
	\bigg(\frac{\pa}{\pa x_\ell}(u_{n+1}^{1/2}u_i) 
	- 3u_i\frac{\pa}{\pa x_\ell}u_{n+1}^{1/2}\bigg) \sum_{k=1}^d P_k(y)
	\bigg(\frac{\pa w_i^\ell}{\pa y_k} + \delta_{k\ell}\bigg)
	\frac{\pa\phi_i^1}{\pa y_k} dydxdt.
\end{align*}
Since the functions $u_i$ are independent of $y$, we see that $w_i^\ell$ 
is in fact a solution of the unit-cell problem
$$
  \diver_y\big(P(y)( \na_y w^\ell_i + e_\ell)\big) = 0\quad\mbox{in }Y, \quad
	\int_Y w_i^\ell(y,t)dy = 0, \quad w_i^\ell\text{ is $Y$-periodic},
$$
where $i=1,\ldots,n$, $\ell=1,\ldots,d$, and 
recalling that $(e_1,\ldots,e_d)$ is the canonical basis of $\R^d$.
These problems do not depend on $i$, so we may set $w^\ell:=w_i^\ell$ for
$i=1,\ldots,n$.

Next, we choose $\phi^1_i=0$ for $i=1,\ldots,n$ in \eqref{wf1}:
\begin{align*}
  0 &= -\int_0^T\int_\Omega u\cdot\pa_t \phi^0 dxdt \\
	&\phantom{xx}{}
	+ \int_0^T\int_\Omega\dashint_{Y}\sum_{i=1}^n P(y)D_i u_{n+1}^{1/2} 
	\big(\na(u_{n+1}^{1/2}u_i) 
	- 3u_i\na u_{n+1}^{1/2} + \na_y W_i\big)\cdot\na \phi_i^0dydxdt.
\end{align*}
Inserting the ansatz \eqref{Wi} and rearranging the terms leads to
\begin{align}\label{macro_eq_nonl}
  0 &= -\int_0^T\int_\Omega u\cdot\pa_t \phi^0 dxdt
	+ \sum_{i=1}^n D_i\int_0^T\int_\Omega\dashint_{Y}  P(y) u_{n+1}^{1/2}
	\bigg\{\na(u_{n+1}^{1/2}u_i) - 3u_i\na u_{n+1}^{1/2} \nonumber \\
	&\phantom{xx}{} + \sum_{\ell=1}^d\bigg(\frac{\pa}{\pa x_\ell}(u_{n+1}^{1/2}u_i)
	- 3u_i\frac{\pa}{\pa x_\ell}u_{n+1}^{1/2}\bigg)\na_y w^\ell\bigg\}\cdot\na_y\phi^0_i
	dydxdt \\
  &= -\int_0^T\int_\Omega u\cdot\pa_t \phi^0 dxdt
	+ \sum_{i=1}^n D_i\int_0^T\int_\Omega  u_{n+1}^{1/2} \sum_{\ell=1}^d  
	\bigg\{\bigg(\frac{\pa}{\pa x_\ell}(u_{n+1}^{1/2}u_i)
	- 3u_i\frac{\pa}{\pa x_\ell}u_{n+1}^{1/2}\bigg)  \nonumber
	\\ &\phantom{xx}{}\times \dashint_{Y} 
	\bigg( P_l(y) \frac{\pa\phi^0_i}{\pa x_\ell}
	+ \sum_{k=1}^d P_k(y) \frac{\pa w^\ell}{\pa y_k}\frac{\pa\phi^0_i}{\pa x_k}\bigg) 
	dy \bigg\} dxdt. \nonumber
\end{align}
Then, defining the macroscopic matrix 
$D_{\rm hom}=(D_{{\rm hom},k\ell})_{k,\ell=1}^d$ by
\begin{equation}\label{Dhom}
  D_{{\rm hom},k\ell} = \dashint_Y P_k(y)\bigg(\delta_{k\ell} 
	+ \frac{\pa w^\ell}{\pa y_k}
	\bigg)dy, \quad  \text{ for } \; \; k,\ell=1,\ldots,d,
\end{equation}
we obtain the macroscopic problem \eqref{macro_2}. We deduce from 
equation \eqref{macro_eq_nonl} and the regularity of $u$ that 
$\partial_t u \in L^2(0,T; H^1(\Omega; \R^n)^\prime)$ and consequently, 
the initial conditions are satisfied in the sense of $H^1(\Omega; \R^n)'$.

In the case of the macroscopic problem \eqref{2.eq} 
with the diffusion matrix~\eqref{1.nonloc} defined in the perforated
domain $\Omega^\eps$, the convergence results of Lemma~\ref{lem.conv_2}
lead to the following two-scale problem:
\begin{align*}
   0 &= -\int_0^T\int_\Omega u\cdot\pa_t\phi \, dxdt
	+ \int_0^T\int_\Omega\dashint_{Y_1}\sum_{i=1}^n u_{n+1}^{1/2}D_i
	\Big(\na(u_{n+1}^{1/2}u_i) + \na_y V_i \\
	&\phantom{xx}{}- 3u_j\big(\na u_{n+1}^{1/2}
	+ \na_y V_{n+1}\big)\Big)\cdot(\na\phi_i^0 + \na_y\phi_i^1)\, dydxdt.
\end{align*}
We can calculate as above to find similar macroscopic equations for the
microscopic problem~\eqref{2.eq} with the only difference that the
unit-cell problem for $\widehat w^\ell$ is given by
\begin{align*}
  & \diver_y(\na_y\widehat w^\ell + e_\ell) = 0\quad\mbox{in }Y_1,\quad
	\int_{Y_1}\widehat w^\ell(y,t) = 0, \\
  & (\na_y\widehat w^\ell + e_\ell)\cdot\nu = 0\quad\mbox{on }\Gamma,
	\quad\widehat w^\ell\mbox{ is $Y$-periodic},
\end{align*}
and the macroscopic diffusion coefficients are 
\begin{equation}\label{Dhom_perfor}
  D_{{\rm hom},k\ell} = \dashint_{Y_1}\bigg(\delta_{k\ell} 
	+ \frac{\pa\widehat w^\ell}{\pa y_k}
	\bigg)dy, \quad \text{ for } \; \;  k,\ell=1,\ldots,d,
\end{equation}
Observe that the specific structure of the microscopic problem implies a
separation of variables in the two-scale problems and that consequently, 
 scalar unit-cell problems determine the macroscopic diffusion matrix.
\end{proof}


\begin{appendix}

\section{Two-scale convergence}\label{app}

We recall the definition and some properties of two-scale convergence. 
Let $\Omega\subset\R^d$ be an open set
and let $Y\subset\R^d$ be the ``periodicity cell'' identified with the 
$d$-dimensional torus with measure $|Y|$. Consider also the perforated domain 
$\Omega^\eps$ and the corresponding subsets $\overline Y_0 \subset Y$ and 
$Y_1 = Y \setminus \overline Y_0$. 

\begin{definition}[Two-scale convergence]
{\rm (i)} A sequence $(u^\eps)$ in
$L^2(\Omega)$ is {\em two-scale convergent} to $u\in L^2(\Omega\times Y)$ 
if for any smooth $Y$-periodic function $\phi:\Omega\times Y\to\R$, 
$$
  \lim_{\eps\to 0}\int_\Omega u^\eps(x)\phi\bigg(x,\frac{x}{\eps}\bigg)dx
	= \int_\Omega\dashint_Y u(x,y)\phi(x,y)dxdy.
$$
{\rm (ii)} The sequence $(u^\eps)$ is {\em strongly two-scale convergent} to
$u\in L^2(\Omega\times Y)$ if 
$$
  \lim_{\eps\to 0}\int_\Omega
	\bigg|u^\eps(x) - u\bigg(x,\frac{x}{\eps}\bigg)\bigg|^2 dx = 0.
$$
\end{definition}

\begin{remark}\rm
Let $[\cdot]^{\sim}$ denote the extension by zero in the domain 
$\Omega\setminus \Omega^\eps$ and $\chi_{\Omega^\eps}$ be the 
characteristic function of $\Omega^\eps$. 
 
{\rm (i)} If $\| u^\eps \|_{L^2(\Omega^\eps)} \leq C$, then 
$\| [u^\eps]^\sim \|_{L^2(\Omega)} \leq C$ and there exists 
$u \in L^2(\Omega\times Y)$ such that, up to a subsequence,  $[u^\eps]^\sim \rightharpoonup \chi_{Y_1} u$ 
two-scale:
\begin{align*}
  \lim_{\eps\to 0}&\int_{\Omega^\eps} u^\eps(x)\phi\bigg(x,\frac{x}{\eps}\bigg)dx 
	= \lim_{\eps\to 0}\int_{\Omega} [u^\eps(x)]^\sim \phi\bigg(x,\frac{x}{\eps}\bigg)
	dx \\ 
  &= \lim_{\eps\to 0}\int_{\Omega} [u^\eps(x)]^\sim \chi_{\Omega^\eps}(x) 
	\phi\bigg(x,\frac{x}{\eps}\bigg)dx 	
	= \int_\Omega\dashint_Y \chi_{Y_1}(y) u(x,y)\phi(x,y)dxdy.
\end{align*}
{\rm (ii)} If $u^\eps \rightharpoonup u$ two-scale with $u \in L^p(\Omega\times Y)$ then 
$$
  u^\eps \rightharpoonup \dashint_Y u(x,y) dy \quad \text{weakly in }L^p(\Omega) 
	\text{ for }p \in [1, \infty).
$$
\end{remark} 

The following results hold. 

\begin{lemma}[\cite{All92, Ngu89}]\label{lem1}
{\rm (i)} If $(u^\eps)$ is bounded in $L^2(\Omega)$, there exists
a subsequence (not relabeled) such that $u^\eps \rightharpoonup u$ two-scale as $\eps\to 0$
for some function $u\in L^2(\Omega\times Y)$. 

{\rm (ii)} If $u^\eps\rightharpoonup u$ weakly in $H^1(\Omega)$ then 
$\na u^\eps \rightharpoonup \na u(x)+\na_y u_1(x,y)$ two-scale, where 
$u_1\in L^2(\Omega;H^1_{\rm per}(Y)/\R)$.
\end{lemma}

\begin{lemma}[\cite{All92}]\label{lem12}
Let $\|u^\eps\|_{L^2(\Omega^\eps)} + \|\nabla u^\eps\|_{L^2(\Omega^\eps)} \leq C$. 
Then, up to a  subsequence, $[u^\eps]^{\sim}$ and $[\nabla u^\eps]^{\sim}$ two-scale converge 
to $\chi_{Y_1}(y) u(x)$ and 
$\chi_{Y_1}(y)[ \nabla u(x) + \nabla_y u_1(x,y)]$ as $\eps \to 0$, respectively, where 
$u \in H^1(\Omega)$ and $u_1 \in L^2(\Omega; H^1_{\rm per}(Y_1)/\mathbb R)$. 
\end{lemma}

\begin{lemma}[\cite{All92, Ngu89}]\label{lem4}
Let $(u^\eps)\subset L^2(\Omega)$ converges two-scale to $u\in L^2(\Omega\times Y)$,
$\|u^\eps\|_{L^2(\Omega)}\to \|u\|_{L^2(\Omega\times Y)}$ as $\eps\to 0$, and let
$(v^\eps)\subset L^2(\Omega)$ converges two-scale to $v\in L^2(\Omega\times Y)$. Then,
as $\eps\to 0$,
$$
  \int_\Omega u^\eps v^\eps dx \to \int_\Omega\dashint_Y u(x,y)v(x,y)dxdy.
$$
\end{lemma}

To define the unfolding operator, let $[z]$ for any $z\in\R^d$ denotes the 
unique  combination $\sum_{i=1}^d k_ie_i$ with $k\in\mathbb Z^d$,   
such that $z-[z]\in Y$, where $e_i$ is the $i$th
canonical basis vector of $\R^d$.

\begin{definition}[\cite{Cio12}] \label{unfold}
Let $p\in[1,\infty]$ and $\phi\in L^p(\Omega)$. Then the unfolding 
operator $\T^\eps$ is defined by $\T^\eps(\phi)\in L^p(\R^d\times Y)$, where
$$
  \T^\eps(\phi)(x,y) = \phi\bigg(\eps \bigg[\frac{x}{\eps}\bigg] + \eps y\bigg) 
	\quad \text{for a.e. }(x,y) \in \Omega \times Y.
$$
Furthermore, for $\psi\in L^p(\Omega^\eps)$, the unfolding operator
$\T^\eps_{Y_1}$ is defined by
$$
  \T^\eps_{Y_1}(\psi)(x,y) = \psi\bigg(\eps \bigg[\frac{x}{\eps}\bigg] + \eps y\bigg) 
	\quad \text{for a.e. }(x,y) \in \Omega \times Y_1.
$$
\end{definition}

For any function $\psi$ defined on $\Omega^\eps$, we have $\T^\eps_{Y_1}(\psi)  
= \T^\eps ([\psi]^{\sim})|_{\Omega \times Y_1}$, whereas  
for $\phi$ defined on $\Omega$, it holds that
$\T^\eps_{Y_1}(\phi|_{\Omega^\eps}) = \T^\eps (\phi)|_{\Omega \times Y_1}$. 
The following result relates the two-scale convergence and the weak convergence
involving the unfolding operator.

\begin{proposition}[\cite{Cio08}]\label{prop.unfold}
Let $(\psi^\eps)$ be a bounded sequence in $L^p(\Omega)$ for some $1<p<\infty$.   
Then the following assertions are equivalent:
\begin{itemize} 
\item[(i)] $(\T^\eps (\psi^\eps))$ converges weakly to $\psi$ in
 $L^p(\Omega\times Y)$.
\item[(ii)] $(\psi^\eps)$ converges two-scale to $\psi$. 
\end{itemize}
\end{proposition}

\begin{theorem} [\cite{Cio12}]\label{conv_unfold}
Let $(\psi^\eps)$ be a bounded sequence in $W^{1, p}(\Omega^\eps)$ for some
$1\leq p < \infty$. Then  there exist functions $\psi \in W^{1, p}(\Omega)$ and 
$\psi_1 \in L^p(\Omega; W^{1, p}_{\rm per}(Y_1)/\R)$ such that as $\eps\to 0$, 
up to a subsequence,  
\begin{align*} 
  & \mathcal T^\eps_{Y_1}(\psi^\eps) \rightharpoonup\psi 
	&& \text{weakly in } L^p(\Omega; W^{1,p}(Y_1)), \\
  & \mathcal T^\eps_{Y_1}(\psi^\eps) \to  \psi 
	&& \text{strongly  in } L^p_{\rm loc}(\Omega; W^{1,p}(Y_1)), \\
  & \mathcal T^\eps_{Y_1}(\na\psi^\eps) \rightharpoonup\na \psi  
	+ \na_y \psi_1 && \text{weakly in } L^p(\Omega\times Y_1). 
\end{align*} 
\end{theorem}

\begin{lemma}[\cite{Cioranescu_book, HJ91}]
{\rm (i)} For $u \in H^1(Y_1)$, there exists an extension $\overline u$ into $Y_0$ 
and thus onto $Y$ such that 
$$
  \|\overline u\|_{L^2(Y)} \leq C \| u\|_{L^2(Y_1)}, \quad  
  \| \na \overline u\|_{L^2(Y)} \leq C \| \na  u\|_{L^2(Y_1)}.
$$
{\rm (ii)} For $u \in H^1(\Omega^\eps)$ there exists an extension $\overline u$ 
into $\Omega$ such that 
$$
  \|\overline u\|_{L^2(\Omega)} \leq C \| u\|_{L^2(\Omega^\eps)}, \quad  
	\| \na \overline u\|_{L^2(\Omega)} \leq C \| \na  u\|_{L^2(\Omega^\eps)}, 
$$
where the constant $C$ is independent of $\eps$.   
\end{lemma} 

\begin{proof}[Sketch of the proof.] 
We can write $u = \dashint_{Y_1} u dy + \psi$, where $\dashint_{Y_1} \psi dy = 0$.  
By standard extension results, we obtain an extension $\overline \psi \in H^1(Y)$ of 
$\psi$.  The definition $\overline u = \dashint_{Y_1} u dy + \overline \psi$ and 
the Poincar\'e inequality imply the results stated in (i).  
The results in (i) and a scaling argument ensure the existence of an extension from $
\Omega^\eps$ into $\Omega$ and estimates in (ii) uniform in $\eps$. 
\end{proof} 

The same results hold also for $u \in W^{1, p}(\Omega^\eps)$,  with $1\leq p < \infty$,
see, e.g., \cite{Acerbi}.

Notice that the corresponding extension operator is linear and continuous from 
$H^1(\Omega^\eps)$ to $H^1(\Omega)$ and by the construction of the extension,  
we have $\overline u = u $ in $\Omega^\eps$. 


\section{Proof of Lemma~\ref{Lemma_entropy}. } \label{sec.lemma}

Consider the entropy density
\begin{equation}\label{entropy_dens}
  h(u) = \sum_{i=1}^{n+1} (u_i\log u_i - u_i + 1)\quad\text{for }  
	u=(u_1,\ldots,u_n) \in \dom,  
\end{equation}
where $u_{n+1}=1-\sum_{i=1}^n u_i$.
Since $h^\prime(u) = (\log(u_1/u_{n+1}), \ldots, \log(u_n/u_{n+1}))$ is invertible 
on $\dom$, the solutions of the microscopic problem are bounded, $u\in\overline{\dom}$.
By Lemma 7 in \cite{ZaJu17}, it holds for all $z\in\R^n$ and $u\in\dom$ that
$$
  z^\top h''(u)A(u)z \ge p_0u_{n+1}\sum_{i=1}^n\frac{z_i^2}{u_i}
	+ \frac{p_0}{2}\frac{1}{u_{n+1}}\bigg(\sum_{i=1}^n z_i\bigg)^2, 
$$
where $p_0 = \min_{i=1,\ldots,n}D_i>0$. 
This shows that for suitable functions $u=(u_1,\ldots,u_n)$,
$$
  \na u:h''(u)A(u)\na u \ge 4p_0u_{n+1}\sum_{i=1}^n|\na u_i^{1/2}|^2
	+ 2p_0|\na u_{n+1}^{1/2}|^2.
$$

The entropy inequality is derived formally from the weak formulation of \eqref{2.eq}
by choosing the test function $w^\eps=h'(u^\eps)$. Since this function is not in
$L^2(0,T;H^1(\Omega))$, we need to consider a regularization. We define
\begin{align*} 
  & w_{\delta}^\eps(u^\eps) = h^\prime(u^\eps_{\delta}) \quad\text{and}\quad 
	\phi^\eps_{\sigma,\delta} = \dfrac{ (u_{n+1}^\eps)^{1/2} }{(u_{n+1}^\eps)^{1/2} 
	+ \sigma }  w^\eps_\delta(u^\eps),  \quad \text{where} \\
  & u^\eps_{\delta, j} = \frac{ u_{j}^\eps +  \delta_1}{1+\delta}, \quad
  u^\eps_{\delta, n+1} = \frac{ u^\eps_{n+1} + \frac  \delta 2} {1+ \delta}
	\quad\text{for } \delta >0, \ \delta_1 = \frac{\delta}{2n}, \ j = 1, \ldots, n.
\end{align*}
Thanks to the regularity properties of $u_i^\eps$, the function 
\begin{align*}
  \na\phi^\eps_{\sigma,\delta,i} 
	&= \frac{ (u_{n+1}^\eps)^{1/2}}{(u_{n+1}^\eps)^{1/2} + \sigma} 
	\bigg(\frac{\na u_i^\eps }{u_i^\eps + \delta_1}
  + \frac{\na u_{n+1}^\eps }{u_{n+1}^\eps + \delta/2 } \bigg) \\
  &\phantom{xx}{}
	+ w_{\delta, i}^\eps(u^\eps) \bigg(\frac{\na(u_{n+1}^\eps)^{1/2}}{(u_{n+1}^\eps)^{1/2} 
	+ \sigma} - \frac{(u_{n+1}^\eps)^{1/2} \na(u_{n+1}^\eps)^{1/2} }{(
	(u_{n+1}^\eps)^{1/2} + \sigma)^2 } \bigg)
\end{align*}
is in $L^2(\Omega_T)$ for each fixed $\sigma$, $\delta >0$. Thus, we can use 
$\phi^\eps_{\sigma, \delta} $ as a test function in \eqref{weak_2}:
\begin{align}\label{weak_regular_2} 
  \sum_{i=1}^n &\int_{\Omega_T} P^\eps(x) D_i (u^\eps_{n+1})^{1/2}
	\big(\na(u^\eps_i(u^\eps_{n+1})^{1/2}) - 3u_i^\eps\na(u^\eps_{n+1})^{1/2}\big)
	\cdot\na\phi^\eps_{\sigma, \delta, i}   dx dt  \nonumber \\
  &{}+ \sum_{i=1}^n \int_0^T \langle \pa_t u^\eps_i,\phi^\eps_{\sigma, \delta, i}
	\rangle dt = 0. 
 \end{align}

The nonnegativity of  $u_{n+1}^\eps$ and $u_j^\eps$  yields the  
pointwise monotone convergences 
\begin{alignat*}{3} 
   &\frac{ (u_{n+1}^\eps)^{1/2}}{(u_{n+1}^\eps)^{1/2} + \sigma}  \to 1, \quad
	&&\frac{u_{n+1}^\eps}{[(u_{n+1}^\eps)^{1/2} + \sigma]^2}  \to 1 \quad
	&&\text{as } \sigma \to 0, \\
  & \frac{ u^\eps_{i} }{ u_i^\eps + \delta/2n} \to 1, \quad
	&& \frac{u^\eps_{n+1}}{ u_{n+1}^\eps + \delta/2} \to 1 \quad
	&& \text{as } \delta \to 0.
\end{alignat*} 
As these four sequences are uniformly bounded by 1, they converge strongly  
in $L^p(\Omega_T)$ for any $1< p < \infty$. 
Thus, the definition of $w_\delta^\eps(u^\eps)$ and the $L^2$-regularity of 
$(u_{n+1}^\eps)^{1/2} \na u_i^\eps$, $(u_{n+1}^\eps)^{1/2} \na (u_i^\eps)^{1/2}$,  
$\na (u_{n+1}^\eps)^{3/2}$, and $\na (u_{n+1}^\eps)^{\frac 12 }$ ensure that
$$
  (u_{n+1}^\eps)^{1/2} \na\phi^\eps_{\sigma, \delta,  i}  
	\to \frac {(u_{n+1}^\eps)^{1/2}\na u_{i}^\eps}{u^\eps_i + \delta/2n} 
	+  \frac{ (u_{n+1}^\eps)^{1/2}\na u_{n+1}^\eps }{u_{n+1}^\eps + \delta/2 }    
	\quad \text{strongly in } L^2(\Omega_T), 
$$
as $\sigma \to 0$, 
and the sequences 
\begin{alignat*}{2} 
  &\frac{u_i^\eps}{u_i^\eps+\delta/2n}  (u_{n+1}^\eps)^{1/2}  \na u_i^\eps
	\cdot \na(u_{n+1}^\eps)^{1/2}, \quad 
  && \frac{u_i^\eps}{u_i^\eps+\delta/2n} u_{n+1}^\eps  |\na (u_i^\eps)^{1/2} |^2, \\
  &\frac{u_{n+1}^\eps}{u_{n+1}^\eps + \delta/2}u_i^\eps|\na (u_{n+1}^\eps)^{1/2}|^2, 
	\quad 
  && \frac{u_{n+1}^\eps}{u_{n+1}^\eps + \delta/2}(u_{n+1}^\eps)^{1/2}\na u_i^\eps 
	\cdot \nabla(u_{n+1}^\eps)^{1/2}
\end{alignat*} 
convergence, up to a subsequence, strongly in $L^1(\Omega_T)$ as $\delta \to 0$, 
for $i=1, \ldots, n$. 
The pointwise convergence of $u_{\delta, j}^\eps$ as $\delta \to 0$ and the
boundedness of the function $s\mapsto s\log s$ for $s\in[0,1]$ ensure the
convergence of $h(u_\delta^\eps)$ in $L^1(\Omega_T)$. 
Rearranging the terms in \eqref{weak_regular_2} and letting first $\sigma \to 0$ 
and then $\delta \to 0$ yields the entropy inequality \eqref{4.epi}.

The same calculations yield entropy estimate for solutions of problem 
\eqref{2.eq}-\eqref{2.bic} with diffusion matrix \eqref{1.nonloc}.


\section{Examples satisfying Assumption A6}\label{app.ex}

We present two cross-diffusion systems whose diffusion matrix and associated
entropy density satisfy Assumption A6.
The first example appears in biofilm modeling. A biofilm is an aggregate of
microorganisms consisting of several subpopulations of bacteria, algae,
protozoa, etc. We assume that the biofilm consist of 
three subpopulations and that it is saturated, i.e., the volume fractions
of the subpopulations $u_i$ sum up to one. Therefore, the volume fraction
of one subpopulation can be expressed by the remaining ones, $u_3=1-u_1-u_2$.
A heuristic approach to define the diffusion fluxes \cite{WED18} 
leads to the cross-diffusion system \eqref{1.eq} with diffusion matrix
$$
  A(u) = \begin{pmatrix}
	D_1(1-u_1) & -D_2u_1 \\ -D_1u_2 & D_2(1-u_2)
	\end{pmatrix},
$$
where $D_1>0$ and $D_2>0$ are some diffusion coefficients. Taking the entropy density
\begin{equation}\label{a.h}
  h(u) = \sum_{i=1}^2 u_i(\log u_i-1) + (1-u_1-u_2)(\log(1-u_1-u_2)-1),
\end{equation}
we compute
$$
  h''(u)A(u) = \begin{pmatrix} D_1/u_1 & 0 \\ 0 & D_2/u_2 \end{pmatrix}. 
$$
This shows that Assumptions A2 and A6 are satisfied with $s_1=s_2=-1/2$.

The second example is a model that describes the evolution of an avascular 
tumor. During the avascular stage, the tumor remains in a diffusion-limited, dormant
stage with a diameter of a few millimeters. We suppose that the tumor growth
can be described by the volume fraction $u_1$ of tumor cells, the volume fraction
of the extracellular matrix $u_2$ (a mesh of fibrous proteins and polysaccharides), 
and the volume fraction of water/nutrients $u_3=1-u_1-u_2$.
Jackson and Byrne \cite{JaBy02} have derived
by a fluiddynamical approach the cross-diffusion model \eqref{1.eq} with
diffusion matrix
$$
  A(u) = \begin{pmatrix}
  2u_1(1-u_1) - \beta\theta u_1u_2^2 & -2\beta u_1u_2(1+\theta u_1) \\
  -2u_1u_2 + \beta\theta(1-u_2)u_2^2 & 2\beta u_2(1-u_2)(1+\theta u_1)
  \end{pmatrix},
$$
where the parameters $\beta>0$ and $\theta>0$ model the strength of the 
partial pressures. With the entropy \eqref{a.h}, we find that \cite[(32)]{JuSt11}
$$
  h''(u)A(u) = \begin{pmatrix}
	2 & 0 \\ \beta\theta u_2 & 2\beta(1+\theta u_1)
	\end{pmatrix}.
$$
Assuming that $\theta<4\sqrt{\beta}$, it follows for $0\le u_1,u_2\le 1$
and $z\in\R^2$ that
$$
  z^\top h''(u)A(u)z \ge (2-\eps)z_1^2 + 2\beta\bigg(1 - \frac{\beta\theta^2}{8\eps}
	\bigg)z_2^2 \ge \kappa|z|^2,
$$
where $\kappa=\min\{2-\eps,2\beta(1-\beta\theta^2/(8\eps)\}>0$ if we choose
$0<\eps<2$. Then Assumption~A2 is fulfilled with $s_1=s_2=0$,
and Assumption~A6 holds as well since $(h''(u)A(u))_{21}$ is bounded from
above by $\beta\theta$. 

\end{appendix}


\end{document}